\documentclass{article}

\usepackage{amsmath,amsfonts,amssymb,amsthm,thm-restate,enumerate,mathtools,color}
\usepackage{hyperref,cleveref}
\usepackage{authblk}
\usepackage{fullpage}

\usepackage{todonotes}

\DeclareMathOperator{\Inf}{Inf}
\DeclareMathOperator{\Stab}{Stab}
\DeclareMathOperator*{\Ex}{\mathbb{E}}

\DeclareMathOperator{\round}{round}
\providecommand{\NAND}{\mathsf{NAND}}
\providecommand{\highlight}[1]{\emph{\textbf{#1}}}

\newtheorem{theorem}{Theorem}[section]
\newtheorem{lemma}[theorem]{Lemma}

\theoremstyle{definition}
\newtheorem{definition}[theorem]{Definition}

\theoremstyle{remark}
\newtheorem{example}[theorem]{Example}

\newif\ifanon
\anonfalse

\title{Approximate Polymorphisms of Predicates}
\ifanon
\author[1]{Anonymous}
\else
\author[1]{Yaroslav Alekseev}
\author[1,2]{Yuval Filmus}
\affil[1]{Taub Faculty of Computer Science, Technion Israel Institute of Technology, Haifa, Israel}
\affil[2]{Faculty of Mathematics, Technion Israel Institute of Technology, Haifa, Israel}
\fi

\begin{document}

\maketitle

\begin{abstract}
A generalized polymorphism of a predicate $P \subseteq \{0,1\}^m$ is a tuple of functions $f_1,\dots,f_m\colon \{0,1\}^n \to \{0,1\}$ satisfying the following property: \emph{If $x^{(1)},\dots,x^{(m)} \in \{0,1\}^n$ are such that $(x^{(1)}_i,\dots,x^{(m)}_i) \in P$ for all $i$, then also $(f_1(x^{(1)}),\dots,f_m(x^{(m)})) \in P$}.

We show that if $f_1,\dots,f_m$ satisfy this property for \emph{most} $x^{(1)},\dots,x^{(m)}$ (as measured with respect to an arbitrary full support distribution $\mu$ on $P$), then $f_1,\dots,f_m$ are close to a generalized polymorphism of $P$ (with respect to the marginals of $\mu$).

Our main result generalizes several results in the literature:
\begin{itemize}
\item Linearity testing (Blum, Luby, and Rubinfeld): $P = \{(0,0,0),(0,1,1),(1,0,1),(1,1,0)\}$.
\item Quantitative Arrow theorems (Kalai; Keller; Mossel): $P = \{x \in \{0,1\}^3 : x \neq (0,0,0),(1,1,1)\}$.
\item Approximate intersecting families (Friedgut and Regev): $P = \{(0,0),(0,1),(1,0)\}$.
\item AND testing (Filmus, Lifshitz, Minzer, and Mossel): $P = \{(0,0,0), (0,1,0), (1,0,0), (1,1,1)\}$.
\item $f$-testing (Chase, Filmus, Minzer, Mossel, and Saurabh): $P = \{(x,f(x)) : x \in \{0,1\}^m\}$.
\end{itemize}
In particular, we extend linearity testing to arbitrary distributions.

We also extend our results to predicates on arbitrary finite alphabets in which all coordinates are ``flexible'' (for each coordinate $j$ there exists $w \in P$ such that $w^{j \gets \sigma} \in P$ for all $\sigma$).
\end{abstract}

\pagebreak

\setcounter{tocdepth}{2}
\tableofcontents

\pagebreak

\section{Introduction}
\label{sec:introduction}

The classical BLR linearity test~\cite{BLR} states that if $f\colon \{0,1\}^n \to \{0,1\}$ satisfies
$\Pr_{x,y \sim \mu_{1/2}}[f(x) \oplus f(y) = f(x\oplus y)] \geq 1-\epsilon$
then there exists a linear function $g\colon \{0,1\}^n \to \{0,1\}$ satisfying $\Pr_{x \sim \mu_{1/2}}[g(x) \neq f(x)] \leq \epsilon$; here $\mu_{1/2}$ is the uniform distribution over $\{0,1\}^n$. Linearity testing is a special case of a more general problem, \emph{approximate polymorphisms}, introduced in~\cite{CFMMS2022}, which unifies several other related results in the literature, including Kalai's quantitative Arrow theorem~\cite{Kalai} and AND~testing~\cite{FLMM2020}.

\begin{definition}[Polymorphism] \label{def:polymorphism}
A function $f\colon \{0,1\}^n \to \{0,1\}$ is a \emph{polymorphism} of a predicate $P \subseteq \{0,1\}^m$ if the following holds: \emph{Given any vectors $x^{(1)},\dots,x^{(m)} \in \{0,1\}^n$ such that $(x^{(1)}_i,\dots,x^{(m)}_i) \in P$ for all $i \in [n]$, also $(f(x^{(1)}), \dots, f(x^{(m)})) \in P$.}

Given a probability distribution $\mu$ on $P$, the function $f$ is a \emph{$(\mu,\epsilon)$-approximate polymorphism of $P$} if
\[
 \Pr_{\substack{(x^{(1)}_1,\dots,x^{(m)}_1) \sim \mu \\ \cdots \\ (x^{(1)}_n,\dots,x^{(m)}_n) \sim \mu}}[(f(x^{(1)}), \dots, f(x^{(m)})) \in P] \geq 1 - \epsilon.
\]
\end{definition}

Using this terminology, we can reformulate the BLR linearity test as follows:
\begin{theorem}[Linearity testing] \label{thm:BLR}
Let $P_{\oplus} = \{(a,b,a \oplus b) : a,b \in \{0,1\}\}$, and let $\mu_\oplus$ be the uniform distribution over $P_\oplus$.

If $f\colon \{0,1\}^n \to \{0,1\}$ is a $(\mu_\oplus,\epsilon)$-approximate polymorphism of $P_\oplus$ then there exists a polymorphism $g\colon \{0,1\}^n \to \{0,1\}$ of $P_\oplus$ such that $\Pr_{x \sim \mu_{1/2}}[g(x) \neq f(x)] \leq \epsilon$.
\end{theorem}

Similarly, Keller's quantitative Arrow theorem~\cite{Keller2012} (which improves on Kalai's result) can be formulated as follows, in the special case of three candidates:
\begin{theorem}[Quantitative Arrow theorem]
\label{thm:keller-arrow}
Let $\mathrm{NAE}_3 = \{(a,b,c) \in \{0,1\}^3 : (a,b,c) \neq (0,0,0),(1,1,1)\}$, and let $\mu$ be the uniform distribution over $\mathrm{NAE}_3$.

If $f\colon \{0,1\}^n \to \{0,1\}$ is a $(\mu,\epsilon)$-approximate polymorphism of $\mathrm{NAE}_3$ then there exists a polymorphism $g\colon \{0,1\}^n \to \{0,1\}$ of $\mathrm{NAE}_3$ such that $\Pr_{x \sim \mu_{1/2}}[g(x) \neq f(x)] = O(\epsilon^{1/3})$.
\end{theorem}

In this paper, we extend \Cref{thm:BLR,thm:keller-arrow} to arbitrary predicates $P \subseteq \{0,1\}^m$ (for all $m$) and to arbitrary distributions $\mu$ on $P$ with full support.

One might hope for a result of the following form: \emph{Any approximate polymorphism of $P$ is close to a polymorphism of $P$.} While this holds for some predicates $P$, it fails for others, as the following counterexample from~\cite{CFMMS2022} demonstrates.

\begin{example}
Let $P_{\NAND} = \{(a, b, \overline{a \land b}) : a,b \in \{0,1\}\}$ and let $\mu_{\NAND}$ be the uniform distribution over $P_{\NAND}$. One can show that the only polymorphisms of $P_{\NAND}$ are dictators: $f(x) = x_i$.

For large $n$, let $f\colon \{0,1\}^n \to \{0,1\}$ be the following function:
\[
 f(x) = \begin{cases}
     x_1 \land x_2 & \text{if } x_1 + \cdots + x_n \leq 0.6 n, \\
     x_1 \lor x_2 & \text{otherwise}.
 \end{cases}
\]
If $x, y \sim \mu_{1/2}(\{0,1\}^n)$ then $x_1 + \cdots + x_n \approx n/2$ and $y_1 + \cdots + y_n \approx n/2$, while $\overline{x_1 \land y_1} + \cdots + \overline{x_n \land y_n} \approx (3/4) n$. Since
\[
 \overline{(x_1 \land x_2) \land (y_1 \land y_2)} = \overline{x_1 \land y_1} \lor \overline{x_2 \land y_2},
\]
the function $f$ is a $(\mu_{\NAND},o(1))$-approximate polymorphism of $P_{\NAND}$. However, $f$ is not close to any exact polymorphism of $P_{\NAND}$.
\end{example}

While we cannot guarantee that an approximate polymorphism of $P$ is close to a polymorphism of $P$, we are able to guarantee that it is close to a \emph{generalized polymorphism} of $P$.

\begin{definition}[Generalized polymorphism] \label{sec:generalized-polymorphism}
A tuple of functions $f_1,\dots,f_m\colon \{0,1\}^n \to \{0,1\}$ is a \emph{generalized polymorphism} of a predicate $P \subseteq \{0,1\}^m$ if the following holds: \emph{Given any vectors $x^{(1)},\dots,x^{(m)} \in \{0,1\}^n$ such that $(x^{(1)}_i,\dots,x^{(m)}_i) \in P$ for all $i \in [n]$, also $(f_1(x^{(1)}), \dots, f_m(x^{(m)})) \in P$.}

Given a probability distribution $\mu$ on $P$, the tuple $f_1,\dots,f_m$ is a \emph{$(\mu,\epsilon)$-approximate generalized polymorphism of $P$} if
\[
 \Pr_{(x^{(1)}_i,\dots,x^{(m)}_i) \sim \mu}[(f_1(x^{(1)}), \dots, f_m(x^{(m)})) \in P] \geq 1 - \epsilon,
\]
where each tuple $(x^{(1)}_i,\dots,x^{(m)}_i)$ is sampled independently according to $\mu$.
\end{definition}

Using this concept, we can state our main theorem:

\begin{restatable}[Main]{theorem}{thmintromain} \label{thm:intro-main}
Let $P \subseteq \{0,1\}^m$ be a non-empty predicate, and let $\mu$ be a distribution on $P$ with full support. For every $\epsilon > 0$ there exists $\delta > 0$ such that the following holds for all $n$.

If $f_1,\dots,f_m\colon \{0,1\}^n \to \{0,1\}$ is a $(\mu,\delta)$-approximate generalized polymorphism of $P$ then there exists a generalized polymorphism $g_1,\dots,g_m\colon \{0,1\}^n \to \{0,1\}$ of $P$ such that
\[
 \Pr_{\mu|_j}[g_j \neq f_j] \leq \epsilon \text{ for all } j \in [m],
\]
where $\mu|_j$ is the marginal distribution of the $j$'th coordinate.
\end{restatable}

In the remainder of the introduction, we first state several other results that we prove in the paper (\Cref{sec:intro-other}), and then review some of the relevant literature (\Cref{sec:intro-literature}).

\subsection{Other results} \label{sec:intro-other}

\Cref{thm:intro-main} has an unspecified dependence between $\delta$ and $\epsilon$; the dependence arising from our current proof is of tower type.
We can improve this dependence in the case of linearity testing.

\begin{restatable}[Linearity testing for general distributions]{theorem}{thmintroblr} \label{thm:intro-blr}
Let $P_{m,b} = \{ (a_1,\dots,a_m) \in \{0,1\}^m : a_1 \oplus \cdots \oplus a_m = b \}$, where $m \ge 3$ and $b \in \{0,1\}$, and let $\mu$ be a distribution on $P_{m,b}$ with full support. The following holds for $\delta = \Theta(\epsilon)$.

If $f_1,\dots,f_m\colon \{0,1\}^n \to \{0,1\}$ is a $(\mu,\delta)$-approximate generalized polymorphism of $P_{m,b}$ then there exists a generalized polymorphism $g_1,\dots,g_m\colon \{0,1\}^n \to \{0,1\}$ of $P_{m,b}$ such that $\Pr_{\mu|_j}[g_j \neq f_j] \leq \epsilon$ for all $j \in [m]$.

Moreover, there exist a set $S \subseteq [m]$ and $b_1,\dots,b_m \in \{0,1\}$ such that $g_j(x) = \bigoplus_{i \in S} x_i \oplus b_j$.

Furthermore, if $f_i = f_j$ and $\mu|_i = \mu|_j$ then $g_i = g_j$.
\end{restatable}

Another difference between \Cref{thm:intro-main} and \Cref{thm:intro-blr} is that we obtain the additional guarantee that if $f_i = f_j$ and the corresponding marginals of $\mu$ coincide, then $g_i = g_j$. We are also able to obtain this guarantee in the case of monotone predicates. The guarantee implies that if all marginals of $\mu$ are identical then an approximate polymorphism is close to an exact polymorphism.

\begin{restatable}[Monotone case]{theorem}{thmintromonotone}
\label{thm:intro-monotone}
Let $P \subseteq \{0,1\}^m$ be a non-empty monotone predicate: if $x \in P$ and $y \leq x$ (pointwise) then $y \in P$. Let $\mu$ be a distribution on $P$ with full support. For every $\epsilon > 0$ there exists $\delta > 0$ such that the following holds for all $n$.

If $f_1,\dots,f_m\colon \{0,1\}^n \to \{0,1\}$ is a $(\mu,\delta)$-approximate generalized polymorphism of $P$ then there exists a generalized polymorphism $g_1,\dots,g_m\colon \{0,1\}^n \to \{0,1\}$ of $P$ such that $\Pr_{\mu|_j}[g_j \neq f_j] \leq \epsilon$ for all $j \in [m]$.
Moreover, $g_j \leq f_j$ pointwise for all $j \in [m]$.

Furthermore, if $f_i = f_j$ and $\mu|_i = \mu|_j$ then $g_i = g_j$.
\end{restatable}

\paragraph{Intersecting families}

Friedgut and Regev~\cite{FR18} proved the following result on almost intersecting families (see also \cite{DT16}).

\begin{restatable}[Friedgut--Regev]{theorem}{thmintrofriedgutregev} \label{thm:intro-friedgut-regev}
Fix $0 < p < 1/2$. For every $\epsilon > 0$ there exist $\delta > 0$ and $j \in \mathbb{N}$ such that the following holds for all $n$ such that $pn$ is an integer.

If $\mathcal{F} \subseteq \binom{[n]}{pn}$ contains a $\delta$-fraction of the edges of the Kneser graph then there exists an intersecting family $\mathcal{G} \subseteq \binom{[n]}{pn}$ depending on $j$ points such that $|\mathcal{F} \setminus \mathcal{G}| \leq \epsilon \binom{n}{pn}$.
\end{restatable}

We provide an alternative proof of \Cref{thm:intro-friedgut-regev} by applying \Cref{thm:intro-monotone} to the NAND predicate. Our proof (which uses the same reduction from $\binom{[n]}{pn}$ to the product measure $\mu_p$ on $\{0,1\}^n$) is arguably simpler.

\paragraph{Polymorphisms over larger alphabets}

So far, we have considered predicates over the binary alphabet. However, the notion of polymorphisms makes sense for every finite alphabet. We conjecture that \Cref{thm:intro-main} extends to this setting. While we are unable to prove this conjecture in full generality, we are able to prove the following special case.

\begin{restatable}[Larger alphabets]{theorem}{thmintroalphabet} \label{thm:intro-alphabet}
Let $\Sigma$ be a finite set, let $P \subseteq \Sigma^m$, and let $\mu$ be a distribution on $P$ with full support. Suppose that for each $j \in [m]$ there exists $w \in P$ such that $w$ remains in $P$ even if we modify its $j$'th coordinate arbitrarily. For every $\epsilon > 0$ there exists $\delta > 0$ such that the following holds for all $n$.

If $f_1,\dots,f_m\colon \Sigma^n \to \Sigma$ is a $(\mu,\delta)$-approximate generalized polymorphism of $P$ then there exists a generalized polymorphism $g_1,\dots,g_m\colon \Sigma^n \to \Sigma$ of $P$ such that $\Pr_{\mu|_j}[g_j \neq f_j] \leq \epsilon$ for all $j \in [m]$.
\end{restatable}

\paragraph{Input/output predicates}

The proofs of \Cref{thm:intro-monotone,thm:intro-alphabet} immediately generalize to a setting which involves two predicates $P,Q$. Given two predicates $P \subseteq \Sigma^m$ and $Q \subseteq \Delta^m$, a tuple $f_1,\dots,f_m\colon \Sigma^n \to \Delta$ is a $(P,Q)$-generalized polymorphism if the following holds: \emph{Given any vectors $x^{(1)},\dots,x^{(m)} \in \Sigma^n$ such that $(x^{(1)}_i,\dots,x^{(m)}_i) \in P$ for all $i \in [n]$, we have $(f_1(x^{(1)}),\dots,f_m(x^{(m)})) \in Q$.}
This setting arises natural in the study of promise CSPs~\cite{AGH17,BG21,BBKO21}, where such polymorphisms are often called \emph{weak polymorphisms}.

In the case of \Cref{thm:intro-monotone}, we require $P$ to be monotone but $Q$ can be arbitrary, and similarly, in the case of \Cref{thm:intro-alphabet}, the stated condition need only hold for $P$. While we do not work out these generalizations explicitly, they follow immediately from the proofs. In contrast, the proof of \Cref{thm:intro-main} does rely on the assumption $P = Q$.

\subsection{Related work} \label{sec:intro-literature}

\paragraph{Arrow's theorem}

Arrow's celebrated theorem~\cite{Arrow} can be expressed in the language of polymorphisms. Let $m \geq 3$. For every permutation $\pi \in S_m$, let $I(\pi) \in \{0,1\}^{\binom{m}{2}}$ be the following vector: $I(\pi)(i,j) = [\pi(i) > \pi(j)]$. Arrow's theorem states that if $f_{i,j}$ is a generalized polymorphism of $P_m := \{ I(\pi) : \pi \in S_m \}$ and each $f_{i,j}$ is \emph{unanimous} (satisfies $f_{i,j}(b,\dots,b) = b$ for $b \in \{0,1\}$) then there exists $k \in [n]$ such that $f_{i,j}(x) = x_k$ for all $i,j$.

Kalai~\cite{Kalai} considered the case $m = 3$. He showed that if $f_{1,2},f_{2,3},f_{3,1}$ is a $(\mu,\epsilon)$-approximate polymorphism of $P_3$, where $\mu$ is the uniform distribution, and furthermore the $f_{i,j}$ are \emph{balanced} (satisfy $\Pr_{\mu_{1/2}}[f_{i,j} = 1] = 1/2$), then there exists $k \in [n]$ such that either $f_{i,j}(x) = x_k$ for all $i,j$, or $f_{i,j} = 1 - x_k$ for all $i,j$. In this case the predicate $P_3$ consists of all triples $(a,b,c)$ such that $a,b,c$ are not all equal.
Keller~\cite{Keller2010} extended Kalai's result to arbitrary $m \ge 3$ and to arbitrary distributions, under various (severe) restrictions on the $f_{i,j}$.

Mossel~\cite{Mossel2012} determined all generalized polymorphisms of $P_m$ for all $m \ge 3$ (without assuming unanimity), and proved that approximate generalized polymorphisms of these predicates (with respect to the uniform distribution) are close to exact generalized polymorphisms. His techniques in fact work for arbitrary distributions (the missing piece, reverse hypercontractivity for arbitrary distributions, was proved in~\cite{MOS13}).
Keller~\cite{Keller2012} improved on Mossel's result in the case of the uniform distribution by determining the optimal dependence between $\epsilon$ and $\delta$.

\paragraph{Linearity testing}

Blum, Luby and Rubinfeld~\cite{BLR} were the first to propose the BLR test. They analyzed it using self-correction. Bellare et al.~\cite{BCHKS96} later gave a different argument using Fourier analysis. The test was generalized to arbitrary prime fields by Kiwi~\cite{Kiwi03}.

David et al.~\cite{DDGKS17} extended the BLR test to the setting of constant weight inputs, which is analogous to the setting of \Cref{thm:intro-friedgut-regev}. Translated to the setting of \Cref{thm:intro-blr}, they extended the BLR test to a natural distribution $\mu$ with $\Pr[\mu|_i = 1] = p$ for all $i$. Dinur et al.~\cite{DFH2025} analyzed the related ``affine test'' $f(x \oplus y \oplus z) = f(x) \oplus f(x \oplus y) \oplus f(x \oplus z)$ for the same distribution. Both of these works used \emph{agreement theorems} to reduce the biased case to the unbiased case: David et al.\ used the agreement theorem of Dinur and Steurer~\cite{DS14}, and Dinur et al.\ proved their own agreement theorem, which we also use to prove \Cref{thm:intro-blr}.

\paragraph{Functional predicates}

Filmus et al.~\cite{FLMM2020}, prompted by work of Nehama~\cite{Nehama} on judgment aggregation, proved \Cref{thm:intro-main} for the predicate $P_\land = \{ (a,b,a\land b) : a,b \in \{0,1\} \}$ and various distributions. Their analysis combined Bourgain's tail bound~\cite{Bourgain02,KKO18} with a study of the ``one-sided noise operator''.

Chase et al.~\cite{CFMMS2022} proved \Cref{thm:intro-main} for all predicates of the form $P_f = \{ (a_1,\dots,a_m, f(a_1,\dots,a_m)) : a_1,\dots,a_m \in \{0,1\} \}$ for an arbitrary function $f\colon \{0,1\}^m \to \{0,1\}$, with respect to the uniform distribution. They asked whether \Cref{thm:intro-main} extends to arbitrary predicates, a question we answer in the affirmative in this paper. Their general approach was similar to the approach taken in this paper, combining Jones' regularity lemma with the It Ain't Over Till It's Over theorem.

\subsection*{Paper organization}

We give an outline of our proof techniques in \Cref{sec:outline}. We prove \Cref{thm:intro-monotone} in \Cref{sec:monotone}, \Cref{thm:intro-main} in \Cref{sec:main}, \Cref{thm:intro-blr} in \Cref{sec:linearity}, \Cref{thm:intro-friedgut-regev} in \Cref{sec:intersecting}, and \Cref{thm:intro-alphabet} in \Cref{sec:general-alphabets}. Our proofs require several versions of Jones' regularity lemma~\cite{Jones2016}, proved in \Cref{sec:regularity}. We close the paper in \Cref{sec:open-questions} with a few open questions.

\paragraph{Acknowledgments}
This research was supported by ISF grant no.\ 507/24. 

We thank Elchanan Mossel for pointing us to his unpublished work with Schramm~\cite{MosselSchramm2009}.

We thank Yumou Fei for pointing out a mistake in the proof of \Cref{thm:intro-monotone}. The original proof attempted to use a more efficient version of Jones' regularity lemma, which outputs a decision tree rather than a junta (cf.~\cite{Mossel20}), in order to obtain better parameters.

\section{Proof outline} \label{sec:outline}

In this section we give a brief overview of the proof of \Cref{thm:intro-main}. We start with the proof in the monotone case (corresponding to \Cref{thm:intro-monotone}), which is more intuitive, and then describe the general case.

\subsection{Monotone case}
\label{sec:outline-monotone}

\paragraph{Triangle removal lemma}

The proof of \Cref{thm:intro-monotone} has the same general outline as the proof of the triangle removal lemma~\cite{RSz78} using Sz\'emeredi's regularity lemma~\cite{Sz78}.

\begin{theorem}[Triangle removal lemma]
For every $\epsilon > 0$ there exists $\delta > 0$ such that the following holds for all $n$.

If $G$ is a graph on $n$ vertices with at most $\delta \binom{n}{3}$ triangles, then we can make $G$ triangle-free by removing at most $\epsilon \binom{n}{2}$ edges.
\end{theorem}

The proof requires Sz\'emeredi's regularity lemma, which we state in a qualitative fashion.

\begin{theorem}[Regularity lemma for graphs]
For every $\epsilon > 0$ there is $M > 1/\epsilon$ such that the following holds for all graphs $G$.

The vertex set of $G$ can be partitioned into $M$ parts $V_1,\dots,V_M$ of almost equal size such that all but an $\epsilon$-fraction of pairs $(i,j)$  are \emph{$\epsilon$-regular}: there exists $p_{i,j} \in [0,1]$ such that the edges of $G$ connecting $V_i$ to $V_j$ ``behave like'' a random bipartite graph with edge density $p_{i,j}$, up to an error $\epsilon$.
\end{theorem} 

Given the regularity lemma, the proof of the triangle removal lemma is quite simple. Given a graph $G$, we apply the regularity lemma with an appropriate parameter $\eta > 0$, obtaining a partition $V_1,\dots,V_M$ in which all but an $\eta$-fraction of pairs is $\eta$-regular. We remove all edges between $V_i$ and $V_j$ if either (i) $i = j$, or (ii) the pair $(i,j)$ is not $\eta$-regular, or (iii) $p_{i,j} \leq \eta$. In total, we have removed at most an $(1/M+2\eta)$-fraction of edges. Choosing $\eta$ so that $\eta = \epsilon/3$, this fits within our budget.

For an appropriate choice of $\delta > 0$, the resulting graph is triangle-free. Indeed, any remaining triangle must involve three different parts $V_i,V_j,V_k$. Since the triangle survived the pruning process, we must have $p_{i,j},p_{i,k},p_{j,k} \geq \eta$. The \emph{triangle counting lemma} shows that the subgraph of $G$ induced by $V_i,V_j,V_k$ contains $\Omega(\eta^3)$ of the possible triangles between these vertices, and so $G$ contains $\Omega((\eta/M)^3) \binom{n}{3}$ triangles. Choosing $\delta$ appropriately, this contradicts the assumption on $G$.

\paragraph{Approximate polymorphisms of monotone predicates}

The proof of \Cref{thm:intro-monotone} follows the same plan. Sz\'emeredi's regularity lemma is replaced by Jones' regularity lemma~\cite{Jones2016}.\footnote{The same regularity lemma had been proved earlier by Mossel and Schramm~\cite{MosselSchramm2009} and (independently) by O'Donnell, Servedio, Tang, and Wan~\cite{ODSTW2010}, but Jones' paper is the only one which is publicly available. Moseel~\cite{Mossel20} proved a version of the regularity lemma for several functions at once.}

\begin{theorem}[Jones' regularity lemma]
For every $\epsilon,\tau > 0$, $d \in \mathbb{N}$, and $p \in (0,1)$ there exists $M \in \mathbb{N}$ such that the following holds for all $n$ and all functions $f\colon \{0,1\}^n \to \{0,1\}$.

There exists a set $J \subseteq [n]$ of size at most $M$ such that
\[
 \Pr_{x \sim \mu_p}[f|_{J \gets x} \text{ is not $(d,\tau)$-regular with respect to $\mu_p$}] \leq \epsilon,
\]
where $\mu_p$ is the product distribution on $\{0,1\}^J$ with $\Pr[x_j = 1] = p$, and a function $g\colon \{0,1\}^{J^c} \to \{0,1\}$ is $(d,\tau)$-regular if $\Inf_i[g^{\leq d}] \leq \tau$ for all $i \in J^c$, where the influence is computed with respect to $\mu_p$.
\end{theorem}

Jones' lemma is proved by a simple potential function argument. In applications we need a variant of this lemma for several functions, and also allowing an initial set as a starting point. We also need two extensions of the lemma, one to functions $f\colon \{0,1\}^n \to [0,1]$, for our alternative proof of \Cref{thm:intro-friedgut-regev}, and another to functions $f\colon \Sigma^n \to \Sigma$, for proving \Cref{thm:intro-alphabet}. 
We prove all of these different versions of Jones' regularity lemma in \Cref{sec:regularity}.

The notion of regularity promised by Jones' regularity lemma is geared toward a result known as \emph{It Ain't Over Till It's Over}~\cite{MOO10}, which is key to our counting lemma.

\begin{theorem}[It Ain't Over Till It's Over]
For every $p,q \in (0,1)$ and $\epsilon > 0$ there exist parameters $d \in \mathbb{N}$ and $\tau,\delta > 0$ such that the following holds for all $n$ and all functions $f\colon \{0,1\}^n \to \{0,1\}$ which are $(d,\tau)$-regular with respect to $\mu_p$.

Let $\rho$ be a random restriction obtained as follows: for each coordinate independently, leave it free with probability $q$, and otherwise sample it according to $\mu_p$. If $\Ex_{\mu_p}[f] \geq \epsilon$ then
\[
 \Pr_\rho[\Ex_{\mu_p}[f|_\rho] \geq \delta] \geq 1 - \epsilon.
\]
\end{theorem}

Applying the theorem to both $f$ and $1-f$, it implies that if $f$ is regular and its expectation is in $[\epsilon,1-\epsilon]$, then even if we sample all but a $q$-fraction of inputs, its expectation still lies in $[\delta,1-\delta]$ (where $\delta$ could be much smaller than $\epsilon$), with probability $1-\epsilon$. This explains its moniker.

\medskip

We outline the proof of \Cref{thm:intro-monotone} in the special case of the predicate $P_{\NAND} = \{(0,0), (1,0), (0,1)\}$; the general case involves no further complications. Let us recall what we would like to prove.

\begin{theorem}[NAND testing]
Let $\mu$ be a distribution on $P_{\NAND}$ with full support.
For every $\epsilon > 0$ there exists $\delta > 0$ such that the following holds for all $n$.

If $f_1,f_2\colon \{0,1\}^n \to \{0,1\}$ is a $(\mu,\delta)$-approximate generalized polymorphism of $P_{\NAND}$ then there exists a generalized polymorphism $g_1,g_2\colon \{0,1\}^n \to \{0,1\}$ of $P_{\NAND}$ such that
\[
 \Pr_{\mu|_j}[g_j \neq f_j] \leq \epsilon \text{ for } j \in \{1, 2\}.
\]
Furthermore, if $f_1 = f_2$ and $\mu|_1 = \mu|_2$ then $g_1 = g_2$.
\end{theorem}

Suppose we are given $f_1,f_2$. Apply Jones' regularity lemma with $p := \mu|_1$, $\epsilon := \epsilon/4$, and appropriate $d,\tau$ to $f_1$ and with $p := \mu|_2$ to $f_2$, obtaining a set $J$ of size at most $M$, for a parameter $M$ depending on $\mu,\epsilon,d,\tau$. Then
\[
 \Pr_{(x^{(1)},x^{(2)}) \sim \mu^J}[f_1|_{J \gets x^{(1)}} \text{ and } f_2|_{J \gets x^{(2)}} \text{ are $(d,\tau)$-regular}] \geq 1 - \epsilon/2.
\]

We define $g_1,g_2$ by zeroing outputs of $f_1,f_2$ as follows. We set $g_j|_{J \gets x^{(j)}} \equiv 0$ if either (i) $f_j|_{J \gets x^{(j)}}$ is not regular, or (ii) $\Ex[f_j|_{J \gets x^{(j)}}] \leq \epsilon/2$. By construction, $\Pr[g_j \neq f_j] \leq \epsilon$. Also, if $f_1 = f_2$ and $\mu|_1 = \mu|_2$ then $g_1 = g_2$.

For an appropriate choice of $\delta$, the resulting pair $(g_1,g_2)$ is a generalized polymorphism of $P_\NAND$. Indeed, if this is not the case, then there exist some $x^{(1)},x^{(2)} \in \{0,1\}^J$ such that $g_1|_{J \gets x^{(1)}},g_2|_{J \gets x^{(2)}}$ are both non-zero. If this happens then $g_1|_{J \gets x^{(1)}},g_2|_{J \gets x^{(2)}}$ are both $(d,\tau)$-regular and their expectations are at least $\epsilon/2$. We will prove a counting lemma which states that for an appropriate choice of $d,\tau$,
\[
 \Pr_{(y^{(1)},y^{(2)}) \sim \mu^{J^c}}[g_1|_{J \gets x^{(1)}}(y^{(1)}) = g_2|_{J \gets x^{(2)}}(y^{(2)}) = 1] \geq \gamma,
\]
where $\gamma$ depends on $\mu$ and $\epsilon$. Since $g_1,g_2$ agree with $f_1,f_2$ on these inputs, this implies that
\[
 \Pr_{(z^{(1)},z^{(2)}) \sim \mu^n}[f_1(z^{(1)}) = f_2(z^{(2)}) = 1] \geq \min(\mu)^M \gamma, \text{ where } \min(\mu) = \min_{w \in P_\NAND} \mu(w).
\]
Choosing $\delta$ to be smaller than the right-hand side, we obtain a contradiction to the assumption that $(f_1,f_2)$ is a $(\mu,\delta)$-approximate generalized polymorphism.

\smallskip

It remains to prove the counting lemma.

\begin{lemma}[Counting lemma for NAND]
Let $\mu$ be a distribution on $P_\NAND$ with full support.
For every $\epsilon > 0$ there exist $d \in \mathbb{N}$ and $\tau,\gamma > 0$ such that the following holds for all $n$.

If $\phi_1,\phi_2\colon \{0,1\}^n \to \{0,1\}$ are $(d,\tau)$-regular (with respect to $\mu|_j$) and have expectation at least $\epsilon/2$ (with respect to $\mu|_j$) then
\[
 \Pr_{(y^{(1)},y^{(2)}) \sim \mu^n}[\phi_1(y^{(1)}) = \phi_2(y^{(2)}) = 1] \geq \gamma.
\]
\end{lemma}

In order to prove this lemma, we will sample $\mu$ in two steps. In the first step, we sample a restriction $\rho$:
\[
 \rho =
 \begin{cases}
     (0,0) & \text{w.p. } (1 - q)\mu(0,0) - q, \\
     (1,0) & \text{w.p. } (1 - q)\mu(1,0), \\
     (0,1) & \text{w.p. } (1 - q)\mu(0,1), \\
     (\ast, 0) & \text{w.p. } q, \\
     (0, \ast) & \text{w.p. } q.
 \end{cases}
\]
We choose the parameter $q > 0$ to be such that all probabilities are positive.

In the second step, if we sampled $(\ast, 0)$, then we sample the first coordinate using $\mu|_1$. Similarly, if we sampled $(0, \ast)$, then we sample the second coordinate using $\mu|_2$. The end result has the same distribution as $\mu$ by design. Also, the distribution of $\rho|_1$ given that $\rho|_1 \neq \ast$ is the same as $\mu|_1$, and similarly for $\rho|_2$.

We apply It Ain't Over Till It's Over to the function $\phi_1$ with $\epsilon := \epsilon/2$ and $q := q$ to obtain $d_1,\tau_1,\delta_1$, and to the function $\phi_2$ to obtain $d_2,\tau_2,\delta_2$. Choosing $d = \max(d_1,d_2)$ and $\tau = (\tau_1,\tau_2)$, It Ain't Over Till It's Over implies that
\[
 \Pr_{(u^{(1)},u^{(2)}) \sim \rho^n}[\Ex_{\mu|_1}[\phi_1|_{\rho|_1}] \geq \delta_1 \text{ and } \Ex_{\mu|_2}[\phi_2|_{\rho|_2}] \geq \delta_2] \geq 1 - \epsilon.
\]
After apply the restriction $\rho$, the events $\phi_1|_{\rho|_1} = 1$ and $\phi_2|_{\rho|_2} = 1$ are independent.
It follows that
\[
 \Pr_{(y^{(1)},y^{(2)}) \sim \mu^n}[\phi_1(y^{(1)}) = \phi_2(y^{(2)}) = 1] \geq (1-\epsilon) \delta_1 \delta_2.
\]
Setting $\gamma := (1-\epsilon) \delta_1 \delta_2$ completes the proof.

\subsection{General case}
\label{sec:outline-general}

In the monotone case, we obtained $g_j$ from $f_j$ by zeroing out certain ``subfunctions'' $f_j|_{J \gets x^{(j)}}$. In the general case, we also need to fix some subfunctions to one. This suggests the following counting lemma, for a given predicate $P \subseteq \{0,1\}^m$.

\begin{lemma}[Counting lemma for $P$]
Let $\mu$ be a distribution on $P$ with full support. For every $\epsilon > 0$ there exist $d \in \mathbb{N}$ and $\tau,\gamma > 0$ such that the following holds for all $n$.

Let $\phi_1,\dots,\phi_m\colon \{0,1\}^n \to \{0,1\}$ be functions such that $\phi_j$ is $(d,\tau)$-regular with respect to $\mu|_j$ for all $j \in [m]$. Define a function $\chi_\epsilon\colon [m] \to \{0,1,*\}$ as follows:
\[
 \chi_\epsilon(j) =
 \begin{cases}
     0 & \text{if } \Ex_{\mu|_j}[\phi_j] \leq \epsilon, \\
     1 & \text{if } \Ex_{\mu|_j}[\phi_j] \geq 1-\epsilon, \\
     * & \text{otherwise}.
 \end{cases}
\]
Let $\alpha\colon [m] \to \{0,1\}$ be any assignment consistent with $\chi_\epsilon$. Then
\[
 \Pr_{(y^{(1)},\dots,y^{(m)}) \sim \mu^n}[(\phi_1(y^{(1)}), \dots, \phi_m(y^{(m)})) = \alpha] \geq \gamma.
\]
\end{lemma}

Given such a counting lemma, we could hope to complete the proof as in the monotone setting, this time obtaining $g_j$ from $f_j$ by setting subfunctions to constants according to $\chi_\epsilon$. There are two main issues with this plan:
\begin{enumerate}
    \item The counting lemma does not hold for all predicates $P$.

    A simple example is the predicate $P_= = \{(0,0),(1,1)\}$ with respect to the uniform distribution. If $\phi_1 = \phi_2$ then the lemma holds only for $\alpha \in \{(0,0),(1,1)\}$. The same problem occurs for $P_{\neq} = \{(0,1),(1,0)\}$.

    A more complicated example is the predicate $P_\oplus = \{(a,b,a\oplus b) : a,b \in \{0,1\}\}$ with respect to the uniform distribution. If $\phi_1(x) = \phi_2(x) = \phi_3(x) = x_1 \oplus \cdots \oplus x_n$ (which is $(d,\tau)$-regular for any $d < n$) then the lemma holds only for $\alpha \in P_\oplus$.

    These examples are not surprising: the argument in the monotone case shows that every approximate generalized polymorphism is close to a generalized polymorphism in which moreover each function is a \emph{junta}, implying in particular that every polymorphism is close to a junta. However, this is not the case for the predicates $P_=,P_{\neq},P_\oplus$. 
    
    \item It is not clear how to handle subfunctions which are not regular.

    In the monotone case, it was safe to zero them out, but for general predicates, there is no safe direction.
\end{enumerate}

It turns out that the counting lemma does hold (under the additional assumption that $(\phi_1,\dots,\phi_m)$ is a $(\mu,\gamma)$-approximate generalized polymorphism) as long as $P$ satisfies no affine relations: there is no non-empty set $S$ such that $P|_S = \{ x \in \{0,1\}^S : \bigoplus_{j \in S} x_j = b \}$. If $P$ satisfies the premise of \Cref{thm:intro-alphabet} (for each $j \in [m]$ there is $x^{(j)} \in P$ which remains in $P$ after flipping the $j$'th coordinate) then this can be proved along the lines of the counting lemma for NAND. The general case requires an argument similar to the one used in~\cite{CFMMS2022}, and also uses some ideas from~\cite{Mossel2012}.

In order to handle predicates with affine relations, we first prove \Cref{thm:intro-blr}, a generalization of linearity testing for arbitrary distributions. The proof uses the techniques of~\cite{DFH2025}. \Cref{thm:intro-blr} shows that if $j$ is a coordinate involved in an affine relation, then $f_j$ is close to a (possibly negated) XOR, and we fix $g_j$ to be this XOR. We then remove coordinates from $P$ until all affine relations disappear, enabling us to use the counting lemma.

\smallskip

We handle irregular subfunctions using an approach similar to the proof of the counting lemma for NAND. We find a way to sample $\mu$ in two steps, first sampling a restriction $\rho$ on the coordinates $J^c$ which leaves at most one coordinate free, and then sampling the free coordinate (if any) according to the correct marginal. We show that
\begin{enumerate}[(a)]
\item With constant probability over the choice of $\rho$, if we define $g_j|_{J \gets x^{(j)}}$ by ``rounding'' $f_j|_{J \gets x^{(j)}, J^c \gets \rho}$ according to $\chi_\eta$ (for an appropriate $\eta$) then $(g_1,\dots,g_m)$ is a generalized polymorphism of $P$.
\item For every $(x^{(1)},\dots,x^{(m)}) \in P^J$ such that all $f_j|_{J \gets x^{(j)}}$ are regular, the coloring $\chi_\eta$ (defined according to $f_j|_{J \gets x^{(j)}, J^c \gets \rho}$) is compatible with the coloring $\chi_\epsilon$ (defined according to $f_j|_{J \gets x^{(j)}}$) with probability $1-\epsilon$. 
\end{enumerate}
The second property guarantees that on average $\Pr_{\mu|_j}[g_j \neq f_j] = O(\epsilon)$, allowing us to find a restriction $\rho$ for which both $(g_1,\dots,g_m)$ is a generalized polymorphism of $P$ and $\Pr_{\mu|_j}[g_j \neq f_j] = O(\epsilon)$, completing the proof.

\paragraph{Flexible coordinates}

During the proof of \Cref{thm:intro-main}, we distinguish between two types of coordinates. A coordinate $j \in [m]$ is \emph{flexible} if there exists a partial input, leaving only the $j$'th coordinate unset, both of whose completions belong to the predicate. If no such input exists, then the coordinate is \emph{inflexible}.

If the predicate is monotone and no coordinate is constant, then all of its coordinates are flexible. For functional predicates, which are predicates of the form $\{ (x,f(x)) : x \in \{0,1\}^m \}$ for some $f\colon \{0,1\}^m \to \{0,1\}$, the first $m$ coordinates are flexible, and the final one is not. For the predicates $\{(1,0,0),(0,1,0),(0,0,1)\}$ and $P_{m,b}$ (all $m$-ary vectors with parity $b$), all coordinates are inflexible.

The counting lemma implies that if $(\phi_1,\dots,\phi_m)$ is a regular generalized polymorphism of a predicate without affine relations then $\phi_j$ is almost constant for every inflexible coordinate $j$. This is in fact the only part of the counting lemma which is employed in the rest of the proof.

The counting lemma is proved by induction on $m$ and on the number of coordinates $j$ such that $\phi_j$ is far from constant. It uses the following dichotomy for \emph{almost full predicates}, which are predicates $P$ such that $P|_{[m] \setminus \{j\}} = \{0,1\}^{[m] \setminus \{j\}}$ for all $j$:

\smallskip

\emph{Either one of the coordinates that $P$ depends on is flexible, or $P$ satisfies an affine relation.}

\smallskip

The lack of an analogous property for larger alphabets precludes us from extending \Cref{thm:intro-main} to that setting in full generality. Instead, \Cref{thm:intro-alphabet} assumes that all coordinates are flexible, and its proof is a simplification of the proof of \Cref{thm:intro-main}.

\section{Monotone case} \label{sec:monotone}

In this section we prove \Cref{thm:intro-monotone}.

\thmintromonotone*

The proof closely follows the outline in \Cref{sec:outline-monotone}: we first prove an appropriate counting lemma using It Ain't Over Till It's Over, and then deduce the result using Jones' regularity lemma. Both of these results use the concept of $(d,\tau)$-regularity.

\begin{definition}[Regularity]
Let $d \in \mathbb{N}$, $\tau > 0$, and $p \in (0, 1)$, and recall that $\mu_p$ is the product distribution with $\Pr_{x \sim \mu_p}[x_i = 1] = p$.

A function $f\colon \{0,1\}^n \to \{0,1\}$ is $(d,\tau)$-regular with respect to $\mu_p$ if $\Inf_i[f^{\leq d}] \leq \tau$ for all $i \in [n]$, where
\[
 \Inf_i[f^{\leq d}] = \sum_{\substack{|S| \leq d \\ i \in S}} \hat{f}(S)^2,
\]
and $\hat{f}(S)$ is the Fourier expansion of $f$ with respect to $\mu_p$, that is
\[
 f(x) = \sum_{S \subseteq [n]} \hat{f}(S) \prod_{i \in S} \frac{x_i - p}{\sqrt{p(1-p)}}.
\]
\end{definition}

We already stated It Ain't Over Till It's Over in \Cref{sec:outline-monotone}. Here we restate it with explicit parameters, which can be read from the proof in~\cite{MOO10}.

\begin{theorem}[It Ain't Over Till It's Over] \label{thm:it-aint-over}
For every $p,q \in (0,1)$ and $\epsilon > 0$ the following holds for some constant $C = C(p,q)$ and $d = \Theta(\log(1/\epsilon))$, $\tau = \Theta(\epsilon^C)$, and $\delta = \Theta(\epsilon^C)$.

Let $\rho$ be a random restriction obtained by sampling each coordinate $i \in [n]$ independently according to the following law:
\[
 \rho_i =
 \begin{cases}
    0 & \text{w.p. } (1-p)(1-q), \\
    1 & \text{w.p. } p(1-q), \\
    * & \text{w.p. } q.
 \end{cases}
\]
If $f\colon \{0,1\}^n \to \{0,1\}$ is $(d,\tau)$-regular with respect to $\mu_p$ and $\Ex_{\mu_p}[f] \geq \epsilon$ then
\[
 \Pr_\rho[\Ex_{\mu_p}[f|_\rho] \geq \delta] \geq 1-\epsilon.
\]
\end{theorem}

In order to obtain the best possible parameters, we use Jones' regularity lemma in the following form, essentially proved in~\cite{CFMMS2022}. For definiteness, we reprove it in \Cref{sec:regularity}.

\begin{theorem}[Jones' regularity lemma] \label{thm:jones-was-tree}
For every $\epsilon, \tau > 0$, $m, d \in \mathbb{N}$ and $p_1,\dots,p_m \in (0,1)$ there exists a parameter $M$ for which the following holds.

For all functions $f_1,\dots,f_m\colon \{0,1\}^n \to \{0,1\}$ there exists a set $J$ of size at most $M$ such that for all $j$,
\[
 \Pr_{x \sim \mu_j}[f_j|_{J \gets x} \text{ is $(d,\tau)$-regular with respect to $\mu_{p_j}$}] \geq 1 - \epsilon.
\]
\end{theorem}

\subsection{Counting lemma} \label{sec:monotone-counting-lemma}

We start by stating and proving the counting lemma that we use.

\begin{lemma}[Counting lemma for monotone predicates] \label{lem:monotone-counting}
Let $P \subseteq \{0,1\}^m$ be a non-empty monotone predicate in which no coordinate is constant (i.e., for every $j \in [m]$ there is $x \in P$ with $x_j = 1$). Let $\mu$ be a distribution on $P$ with full support.
For every $\epsilon > 0$ there exists a constant $C = C(P,\mu)$ such that the following holds for $d = \Theta(\log(1/\epsilon))$, $\tau = \Theta(\epsilon^C)$, and $\gamma = \Theta(\epsilon^{Cm})$.

Let $\phi_1,\dots,\phi_m\colon \{0,1\}^n \to \{0,1\}$ be functions such that $\phi_j$ is $(d,\tau)$-regular with respect to $\mu|_j$. Define $\alpha\colon [m] \to \{0,1\}$ as follows:
\[
 \alpha(j) =
 \begin{cases}
     0 & \text{if } \Ex_{\mu|_j}[\phi_j] \leq \epsilon, \\
     1 & \text{otherwise}.
 \end{cases}
\]
If $\alpha \notin P$ then
\[
 \Pr_{(y^{(1)},\dots,y^{(m)}) \sim \mu^n}[(\phi_1(y^{(1)}), \dots, \phi_m(y^{(m)})) \notin P] \geq \gamma.
\]
\end{lemma}

We first prove the lemma in the special case where $P$ is a NAND predicate, and then deduce the general case.

\begin{lemma}[Counting lemma for NAND] \label{lem:monotone-counting-NAND}
Let $m \ge 2$ and let $P_{\NAND} = \{ x \in \{0,1\}^m : x \neq (1,\dots,1) \}$. Let $\mu$ be a distribution on $P_{\NAND}$ with full support.
For every $\epsilon > 0$ there exists a constant $C = C(P,\mu)$ such that the following holds for $d = \Theta(\log(1/\epsilon))$, $\tau = \Theta(\epsilon^C)$, and $\gamma = \Theta(\epsilon^{Cm})$.

Let $\phi_1,\dots,\phi_m\colon \{0,1\}^n \to \{0,1\}$ be functions such that $\phi_j$ is $(d,\tau)$-regular with respect to $\mu|_j$ and $\Ex_{\mu|_j}[\phi_j] \geq \epsilon$. Then
\[
 \Pr_{(y^{(1)},\dots,y^{(m)}) \sim \mu^n}[(\phi_1(y^{(1)}), \dots, \phi_m(y^{(m)})) = (1,\dots,1)] \geq \gamma.
\]
\end{lemma}
\begin{proof}
The proof closely follows the proof of the counting lemma in \Cref{sec:outline-monotone}.

Let $q > 0$ be a small enough parameter. We sample a random restriction $\rho \in (\{0,1,*\}^m)^n$ by sampling each coordinate independently according to the following law, where $p_j = \Pr[\mu|_j = 1]$, $0$ is the zero vector, $e_1 = (1,0,\dots,0)$, $s_1 = (*,0,\dots,0)$, and $e_j,s_j$ are defined analogously by making the special coordinate the $j$'th one:
\begin{itemize}
\item $\rho_i = 0$ with probability $\mu(0) - \sum_j (1-p_j) q$.
\item For each $j$, $\rho_i = e_j$ with probability $\mu(e_j) - p_j q$.
\item For each $w \neq 0,e_1,\dots,e_m$, $\rho_i = w$ with probability $\mu(w)$.
\item For each $j$, $\rho_i = s_j$ with probability $q$.
\end{itemize}

Given $\rho$, we can obtain a sample of $\mu$ by sampling a star in position $j$ according to $\mu|_j$. We denote this distribution by $\mu|_\rho$.

Another crucial property of $\rho$ is that the distribution of $\rho|_j$ given that $\rho|_j \neq *$ is the same as the distribution of $\mu|_j$. Indeed,
\[
 \Pr[\rho|_j = 1 \mid \rho|_j \neq *] = \frac{p_j - p_j q}{1 - q} = p_j.
\]

Apply \Cref{thm:it-aint-over} (It Ain't Over Till It's Over) to each $\phi_j$ with $p := p_j$, $q := q$, $\epsilon := \epsilon$ to obtain $d_j,\tau_j,\delta_j$ such that if $\phi_j$ is $(d_j,\tau_j)$-regular and $\Ex_{\mu|_j}[\phi_j] \geq \epsilon$ then
\[
 \Pr_{\rho}[\Ex_{\mu|_j}[\phi_j|_{\rho|_j}] \geq \delta_j] \geq 1 - \epsilon.
\]

Choose $d = \max(d_1,\dots,d_m)$, $\tau = \min(\tau_1,\dots,\tau_m)$, so that each $\phi_j$ is $(d_j,\tau_j)$-regular. Applying the union bound,
\[
 \Pr_{\rho}[\Ex_{\mu|_j}[\phi_j|_{\rho|_j}] \geq \delta_j \text{ for all } j] \geq 1 - m\epsilon,
\]
and so
\[
 \Pr_{(y^{(1)},\dots,y^{(m)}) \sim \mu^n}[(\phi_1(y^{(1)}), \dots, \phi_m(y^{(m)})) = (1,\dots,1)] \geq (1-m\epsilon) \delta_1 \cdots \delta_m.
\]
The proof concludes by taking $\gamma = (1-m\epsilon) \delta_1 \cdots \delta_m$.
\end{proof}

We prove \Cref{lem:monotone-counting} by applying \Cref{lem:monotone-counting-NAND} to each maxterm of $P$ (recall that a maxterm is $x \notin P$ such that $y \in P$ for all $y \lneq x$).

\begin{proof}[Proof of \Cref{lem:monotone-counting}]
Let $\mathcal{M}$ be the collection of maxterms of $P$. Since no coordinate is constant, each maxterm is of the form $1_S$ for $|S| \geq 2$.

For each $1_S \in \mathcal{M}$, apply \Cref{lem:monotone-counting-NAND} with $\mu = \mu|_S$ and $\epsilon := \epsilon$ to obtain $d_S,\tau_S,\gamma_S$ such that the following holds: if $\phi_j$ is $(d_S,\tau_S)$-regular and $\alpha_j = 1$ for all $j \in S$ then
\[
 \Pr_{(y^{(1)},\dots,y^{(m)}) \sim \mu_n}[\phi_j = 1 \text{ for all } j \in S] \geq \gamma_S.
\]

We choose $d = \max(d_S : 1_S \in \mathcal{M})$,  $\tau = \min(\tau_S : 1_S \in \mathcal{M})$, and $\gamma = \min(\gamma_S : 1_S \in \mathcal{M})$.

If $\alpha \notin P$ then there exists a maxterm $1_S \in \mathcal{M}$ such that $\alpha_j = 1$ for all $j \in S$. \Cref{lem:monotone-counting-NAND} implies that
\[
 \Pr_{(y^{(1)},\dots,y^{(m)}) \sim \mu_n}[\phi_j(y^{(j)}) = 1 \text{ for all } j \in S] \geq \gamma_S \geq \gamma.
\]
Since $1_S$ is a maxterm, this completes the proof.
\end{proof}

\subsection{Main result} \label{sec:monotone-main}

In this section we complete the proof of \Cref{thm:intro-monotone} using Jones' regularity lemma. We first assume that no coordinate of $P$ is constant, and then show how to get rid of this assumption.

\begin{proof}[Proof of \Cref{thm:intro-monotone} assuming no coordinate of $P$ is constant]

Since no coordinate of $P$ is constant, we can apply \Cref{lem:monotone-counting} (the counting lemma) with $\epsilon := \epsilon/2$ to obtain $d,\tau,\gamma$ for which the lemma holds.

We apply \Cref{thm:jones-was-tree} (Jones's regularity lemma) with $\epsilon := \epsilon/2$, $p_1,\dots,p_m$ given by $p_j = \Pr[\mu|_j = 1]$, and the values of $d,\tau$ obtained from \Cref{lem:monotone-counting} to obtain a set $J$ of size at most $M$.

We define the functions $g_j$ as follows. For each setting $x \in \{0,1\}^J$,
\[
 g_j|_{J \gets x}(y) =
 \begin{cases}
     0 & \text{if } f_j|_{J \gets x} \text{ is not $(d,\tau)$-regular or }\Ex_{\mu|_j}[f_j|_{J \gets x}] \leq \epsilon/2, \\
     f_j|_{J \gets x}(y) & \text{otherwise}.
 \end{cases}
\]

If we sample $x$ according to $\mu|_j$ then according to Jones' regularity lemma, $f_j|_{J \gets x}$ is $(d,\tau)$-regular (with respect to $\mu|_j$) with probability at least $1-\epsilon/2$. This shows that $\Pr_{\mu|_j}[g_j \neq f_j] \leq \epsilon/2 + \epsilon/2 = \epsilon$.

It remains to show that $(g_1,\dots,g_m)$ is an extended polymorphism of $P$ for small enough $\delta$. If this is not the case, then there exists an assignment $(x^{(1)},\dots,x^{(m)}) \in P^J$ such that $(g_1|_{J \gets x^{(1)}}, \dots, g_m|_{J \gets x^{(m)}})$ is not an extended polymorphism of $P$.

Apply \Cref{lem:monotone-counting} (the counting lemma) to $\phi_j = f_j|_{J \gets x^{(j)}}$. Observe that for all $(y^{(1)},\dots,y^{(m)}) \in P^{J^c}$,
$(g_1|_{J \gets x^{(1)}}(y^{(1)}),\allowbreak \dots,\allowbreak g_m|_{J \gets x^{(m)}}(y^{(m)})) \leq \alpha$, where $\alpha$ is the assignment defined in the lemma. Since $(g_1|_{J \gets x^{(1)}},\allowbreak \dots,\allowbreak g_m|_{J \gets x^{(m)}})$ is not an extended polymorphism of $P$ and $P$ is monotone, necessarily $\alpha \notin P$. Therefore the lemma shows that
\[
 \Pr_{(y^{(1)},\dots,y^{(m)}) \sim \mu^{J^c}}[(f_1|_{J \gets x^{(1)}}(y^{(1)}),\dots,f_m|_{J \gets x^{(m)}}(y^{(m)})) \notin P] \geq \gamma,
\]
implying that $(f_1,\dots,f_m)$ is not a $(\mu,\delta)$-approximate generalized polymorphism for $\delta = \min(\mu)^M \gamma/2$, where $\min(\mu) = \min_{w \in \mu}(\mu(w))$. Choosing this value of $\delta$ completes the proof.
\end{proof}

The proof of \Cref{thm:intro-monotone} in full generality readily follows. Let $C_0$ be the set of coordinates $j$ such that $x_j = 0$ for all $x \in P$. The premise of the theorem implies that $f_j(0,\dots,0) = 0$ for all $j \in C_0$.

Let $Q$ be the predicate obtained by removing the coordinates in $C_0$. We apply the foregoing to the predicate $Q$, obtaining functions $g_j$ for all $j \notin C_0$. Defining $g_j = f_j$ for all $j \in C_0$ completes the proof.

\section{Main theorem} \label{sec:main}

In this section we prove \Cref{thm:intro-main}.

\thmintromain*

The proof relies on \Cref{thm:intro-blr}, the special case of \Cref{thm:intro-main} for the predicates $P_{m,b} = \{ x \in \{0,1\}^m : x_1 \oplus \cdots \oplus x_m = b \}$ for $m \geq 3$.

The proof also relies on a two-sided version of It Ain't Over Till It's Over, which immediately follows from the one-sided version, \Cref{thm:it-aint-over}, by applying it to both $f$ and $1-f$.

\begin{theorem}[It Ain't Over Till It's Over] \label{thm:it-aint-over-two-sided}
For every $p,q \in (0,1)$ and $\epsilon > 0$ the following holds for some constant $C = C(p,q)$ and $d = \Theta(\log(1/\epsilon))$, $\tau = \Theta(\epsilon^C)$, and $\delta = \Theta(\epsilon^C)$.

Let $\rho$ be a random restriction obtained by sampling each coordinate $i \in [n]$ independently according to the following law:
\[
 \rho_i =
 \begin{cases}
    0 & \text{w.p. } (1-p)(1-q), \\
    1 & \text{w.p. } p(1-q), \\
    * & \text{w.p. } q.
 \end{cases}
\]
If $f$ is $(d,\tau)$-regular with respect to $\mu_p$ and $\epsilon \leq \Ex_{\mu_p}[f] \leq 1-\epsilon$ then
\[
 \Pr_\rho[\delta \leq \Ex_{\mu_p}[f|_\rho] \leq 1-\delta] \geq 1-\epsilon.
\]
\end{theorem}

We also use a junta version of Jones' regularity lemma, which we prove in \Cref{sec:regularity}.

\begin{theorem}[Jones' regularity lemma] \label{thm:jones-junta}
For every $\epsilon,\tau > 0$, $d,m \in \mathbb{N}$, and $p_1,\dots,p_m \in (0,1)$ there exists a function $\mathcal{M}\colon \mathbb{N} \to \mathbb{N}$ such that the following holds for all $n$ and all functions $f\colon \{0,1\}^n \to \{0,1\}$.

For every $J_0 \subseteq [n]$ there exists a set $J \supseteq J_0$ of size at most $\mathcal{M}(|J_0|)$ such that for all $j$,
\[
 \Pr_{x \sim \mu_{p_j}}[f_j|_{J \gets x} \text{ is not $(d,\tau)$-regular}] \leq \epsilon.
\]
\end{theorem}

The proof consists of several parts:
\begin{itemize}
    \item We prove a counting lemma for predicates without affine relations in \Cref{sec:main-counting}.
    \item We outline the rounding procedure in \Cref{sec:main-rounding}.
    \item We prove \Cref{thm:intro-main} under the assumption that no coordinates are constant or duplicate in \Cref{sec:main-generic}.
    \item We prove \Cref{thm:intro-main} in full generality in \Cref{sec:main-conclusion}.
\end{itemize}

\subsection{Counting lemma} \label{sec:main-counting}

A predicate $P \subseteq \{0,1\}^m$ has no \emph{affine relations} if there do not exist a non-empty subset $S$ and $b \in \{0,1\}$ such that $\bigoplus_{i \in S} w_i = b$ for all $w \in P$. In particular, this implies that no coordinate of $P$ is constant, and no two coordinates are always equal or always non-equal.

\begin{lemma}[Counting lemma for predicates without affine relations] \label{lem:counting}
Let $P$ be a predicate without affine relations, and let $\mu$ be a distribution on $P$ with full support. For every $\epsilon > 0$ there exist $d \in \mathbb{N}$ and $\tau,\gamma > 0$ such that the following holds for all $n$.

Let $\phi_1,\dots,\phi_m\colon \{0,1\}^n \to \{0,1\}$ be functions such that $(\phi_1,\dots,\phi_m)$ is a $(\mu,\gamma)$-approximate generalized polymorphism of $P$ and $\phi_j$ is $(d,\tau)$-regular with respect to $\mu|_j$ for all $j \in [m]$. Define a function $\chi_\epsilon\colon [m] \to \{0,1,*\}$ as follows:
\[
 \chi_\epsilon(j) =
 \begin{cases}
     0 & \text{if } \Ex_{\mu|_j}[\phi_j] \leq \epsilon, \\
     1 & \text{if } \Ex_{\mu|_j}[\phi_j] \geq 1-\epsilon, \\
     * & \text{otherwise}.
 \end{cases}
\]

Let $\alpha\colon [m] \to \{0,1\}$ be any assignment consistent with $\chi_\epsilon$. Then $\alpha \in P$ and
\[
 \Pr_{(y^{(1)},\dots,y^{(m)}) \sim \mu^n}[(\phi_1(y^{(1)}), \dots, \phi_m(y^{(m)})) = \alpha] > \gamma.
\]
\end{lemma}

The proof of the case $\chi_\epsilon \equiv *$ will require the following lemma, whose proof is adapted from~\cite{Mossel2012}.

\begin{lemma}[Hitting lemma] \label{lem:hitting}
Let $\mu$ be a distribution over $\{0,1\}^m$ with full support. For every $\epsilon > 0$ there exists $\gamma > 0$ such that the following holds.

Let $\phi_1,\dots,\phi_m\colon \{0,1\}^n \to \{0,1\}$ be functions satisfying $\epsilon \leq \Ex_{\mu|_j}[\phi_j] \leq 1-\epsilon$ for all $j \in [m]$. Then for all $\alpha \in \{0,1\}^m$,
\[
 \Pr_{(y^{(1)},\dots,y^{(m)}) \sim \mu^n}[(\phi_1(y^{(1)}), \dots, \phi_m(y^{(m)})) = \alpha] > \gamma.
\]
\end{lemma}
\begin{proof}
The proof is by induction on $m$. If $m = 1$ then we can take $\gamma = \epsilon/2$, so suppose $m \geq 2$.

Applying the inductive hypothesis with $\mu := \mu|_{\{2,\dots,m\}}$ and $\epsilon := \epsilon$ gives $\gamma_{m-1} > 0$ such that for all $\beta \in \{0,1\}^{m-1}$,
\[
 \Pr_{(y^{(2)},\dots,y^{(m)}) \sim \mu|_{\{2,\dots,m\}}^n}[(\phi_2(y^{(2)}), \dots, \phi_m(y^{(m)})) = \beta] > \gamma_{m-1}.
\]

We now appeal to \cite[Lemma 8.3]{MOS13}, whose statement is as follows. Suppose that $X,Y$ are finite sets and that $\nu$ is a distribution over $X \times Y$ with full support. For every $\eta > 0$ there exists $\zeta > 0$ such that the following holds. If $A \subseteq X^n$ and $B \subseteq Y^n$ have measure at least $\eta$ (with respect to the corresponding marginals of $\nu^n$) then $\nu^n(A \times B) \ge \zeta$.

Given $\alpha \in \{0,1\}^m$, we apply the lemma with $X = \{0,1\}$, $Y = \{0,1\}^{m-1}$, $\nu = \mu$ and $\eta := \min(\epsilon,\gamma_{m-1})$ to obtain $\zeta > 0$. Taking $A = \{y^{(1)} : \phi_1(y^{(1)}) = \alpha_1\}$, $B = \{(y^{(2)},\dots,y^{(m)}) : (\phi_2(y^{(2)}),\dots,\phi_m(y^{(m)})) = (\alpha_2,\dots,\alpha_m)\}$, the lemma implies that
\[
 \Pr_{(y^{(1)},\dots,y^{(m)}) \sim \mu^n}[(\phi_1(y^{(1)}), \dots, \phi_m(y^{(m)})) = \alpha] > \zeta,
\]
completing the proof (taking $\gamma := \zeta/2$).
\end{proof}

We can now prove the counting lemma.

\begin{proof}[Proof of \Cref{lem:counting}]
The proof is by double induction: first on $m$, and then on the set of non-$*$ inputs in $\chi_\epsilon$. We also assume, without loss of generality, that $\epsilon < 1/2$.

If $m = 1$ then necessarily $P = \{0, 1\}$, since otherwise $P$ has a constant coordinate. Therefore the lemma trivially holds with $\gamma = \epsilon$. From now on, we assume that $m \geq 2$, and induct on the number of non-$*$ inputs in $\chi_\epsilon$.

\paragraph{Base case}

Suppose that $\chi_\epsilon(j) = *$ for all $j \in [m]$. We will find $d' \in \mathbb{N}$ and $\tau',\gamma' > 0$ such that whenever $d \geq d'$, $\tau \leq \tau'$ and $\gamma \leq \gamma'$, the assumptions imply that $P = \{0,1\}^m$. Applying \Cref{lem:hitting} with $\mu := \mu$ and $\epsilon := \epsilon$ gives us $\gamma'' > 0$ such that
\[
 \Pr_{(y^{(1)},\dots,y^{(m))} \sim \mu^n}[(\phi_1(y^{(1)}), \dots, \phi_m(y^{(m)})) = \alpha] > \gamma''
\]
for all $\alpha \in \{0,1\}^m$.
Taking $\gamma = \min(\gamma', \gamma'')$ will conclude the proof.

\smallskip

Suppose therefore that $P \neq \{0,1\}^m$. We first show that for appropriate $d',\tau',\gamma'$ there exists an index $j_0 \in [m]$ and two inputs $u,v \in P$ such that $u^{\oplus j_0} \in P$ and $v^{\oplus j_0} \notin P$, where the inputs $u^{\oplus j_0},v^{\oplus j_0}$ are obtained from the inputs $u,v$ by flipping the $j_0$'th coordinate. \highlight{This is where we use the assumption that $P$ has no affine relations.} We then show that the probability that $(\phi_1,\dots,\phi_m)$ evaluates to $v^{\oplus j_0}$ is bounded from below, and obtain a contradiction by setting $\gamma'$ small enough.

For each $j \in [m]$, we apply the inductive hypothesis to the predicate $P|_{[m] \setminus \{j\}}$, the distribution $\mu|_{[m] \setminus \{j\}}$, and $\epsilon := \epsilon$ to obtain $d_j,\tau_j,\gamma_j$ such that the following holds. If for all $k \neq j$, the function $\phi_k$ is $(d_j,\tau_j)$-regular with respect to $\mu|_k$ and satisfies $\epsilon \leq \Ex_{\mu|_k}[\phi_k] \leq 1-\epsilon$, then $\alpha \in P|_{[m] \setminus \{j\}}$ for every $\alpha \in \{0,1\}^{[m] \setminus \{j\}}$, implying that $P|_{[m] \setminus \{j\}} = \{0,1\}^{[m] \setminus \{j\}}$.

Choose $d''' = \max(d_1,\dots,d_m)$, $\tau''' = \min(\tau_1,\dots,\tau_m)$, $\gamma''' = \min(\gamma_1,\dots,\gamma_m)$. If $d \geq d'''$, $\tau''' \leq \tau$, $\gamma''' \leq \gamma$ then the assumptions of the lemma imply that $P|_{[m] \setminus \{j\}} = \{0,1\}^{[m] \setminus \{j\}}$ for all $j$. This implies that for every $w \in \{0,1\}^m$ and every $j \in [m]$, either $w \in P$ or $w^{\oplus j} \in P$ (or both).

Let $J$ be the set of variables that $P$ depends on: there is $w \in P$ such that $w^{\oplus j} \notin P$. By assumption, $\emptyset \neq P \neq \{0,1\}^m$, and so $J$ is non-empty. If for every $w \in P$ and every $j \in J$ we have $w^{\oplus j} \notin P$ then $P|_J$ consists of all inputs with a given parity, contradicting the assumption that $P$ has no affine relations. Therefore there must exist $j_0 \in J$ and $u \in P$ such that $u^{\oplus j_0} \in P$. Since $j_0 \in J$, there also exists $v \in P$ such that $v^{\oplus j_0} \notin P$.

\smallskip

The remainder of the argument is reminiscent of the corresponding argument in~\cite{CFMMS2022}. As in the proof of the counting lemma for NAND (\Cref{lem:monotone-counting-NAND}), we will sample $\mu$ using a two-step process. Let $u_\ast$ be obtained from $u$ by setting the $j_0$'th coordinate to $*$. First, we sample a restriction $\rho \in (\{0,1\}^m \cup \{u_\ast\})^n$ by sampling each coordinate independently according to the following law, where $q = \min(\mu(u), \mu(u^{\oplus j_0}))$:
\begin{itemize}
    \item $\rho_i = u$ with probability $\mu(u) - \Pr[\mu|_{j_0} = u_{j_0}] q$.
    \item $\rho_i = u^{\oplus j_0}$ with probability $\mu(u^{\oplus j_0}) - \Pr[\mu|_{j_0} = u^{\oplus j_0}_{j_0}] q$.
    \item $\rho_i = w$ with probability $\mu(w)$ for any $w \neq u,u^{\oplus j_0}$.
    \item $\rho_i = u_*$ with probability $q$.
\end{itemize}

Given $\rho$, we can obtain a sample of $\mu$ by sampling the $j_0$'th coordinate according to $\mu|_{j_0}$ if $\rho_i = u_*$. As in \Cref{lem:monotone-counting-NAND}, the distribution of $\rho|_j$ given that $\rho|_j \neq *$ is the same as $\mu|_j$.

Apply \Cref{thm:it-aint-over-two-sided} (It Ain't Over Till It's Over) to $\phi_{j_0}$ with $p := \Pr[\mu|_{j_0} = 1]$, $q := q$, $\epsilon := \epsilon$ to obtain $d'''',\tau'''',\delta$, and take $d'' = \max(d''',d'''')$ and $\tau'' = \min(\tau''',\tau'''')$. The assumptions of the lemma imply that
\[
 \Pr_\rho\bigl[\Pr_{\mu|_{j_0}}[\phi_{j_0}|_{\rho|_{j_0}} = v^{\oplus j_0}_{j_0}] \geq \delta\bigr] \geq 1-\epsilon.
\]

Since $P|_{[m] \setminus \{j_0\}} = \{0,1\}^{[m] \setminus \{j_0\}}$, we can apply \Cref{lem:hitting} to $\mu|_{[m] \setminus \{j_0\}}$ and $\epsilon := \epsilon$ to obtain $\gamma'''' > 0$ such that
\[
 \Pr_{y \sim (\mu|_{[m] \setminus \{j_0\}})^n}[\phi_j(y^{(j)}) = v_j \text{ for all } j \neq j_0] > \gamma''''.
\]
It follows that
\[
 \Pr_{(y^{(1)},\dots,y^{(m)}) \sim \mu^n}[(\phi_1(y^{(1)}),\dots,\phi_m(y^{(m)})) = v^{\oplus j_0}] > (1-\epsilon) \delta \gamma''''.
\]
Setting $\gamma' = \min(\gamma''',(1-\epsilon)\delta\gamma'''')$ contradicts the assumption that $(\phi_1,\dots,\phi_m)$ is a $(\mu,\gamma)$-approximate generalized polymorphism of $P$ for some $\gamma \leq \gamma'$, and we conclude that necessarily $P = \{0,1\}^m$, as claimed above. This concludes the proof of the base case.

\paragraph{Inductive case}
Suppose that $\chi_\epsilon(j_0) \neq *$ for some $j_0 \in [m]$.

Let us try to reduce the statement of the lemma to the same statement for $P|_{[m] \setminus \{j_0\}}$. We apply the inductive hypothesis to the predicate $P|_{[m] \setminus \{j_0\}}$, the distribution $\mu|_{[m] \setminus \{j_0\}}$, and $\epsilon := \epsilon$ to obtain $d',\tau',\gamma'$ such that if $d \geq d'$, $\tau \leq \tau'$, $\gamma \leq \gamma'$ then the assumptions imply that for every $\alpha$ consistent with $\chi_\epsilon$, $\alpha|_{[m] \setminus \{j_0\}} \in P|_{[m] \setminus \{j_0\}}$, and moreover
\[
 \Pr_{y \sim (\mu|_{[m] \setminus \{j_0\}})^n}[\phi_j(y^{(j)}) = \alpha_j \text{ for all } j \neq j_0] > \gamma'.
\]

If $\Pr_{\mu|_{j_0}}[\phi_{j_0} = \alpha_{j_0} \oplus 1] \leq \gamma'/2$ then
\[
 \Pr_{(y^{(1)},\dots,y^{(m)}) \sim \mu^n}[(\phi_1(y^{(1)}),\dots,\phi_m(y^{(m)})) = \alpha] > \gamma'/2,
\]
implying that the lemma holds if $\gamma \leq \gamma'/2$; we automatically get that $\alpha \in P$ since $(\phi_1,\dots,\phi_m)$ is a $(\mu,\gamma'/2)$-approximate generalized polymorphism of $P$.

If $\Pr_{\mu|_{j_0}}[\phi_{j_0} = \alpha_{j_0} \oplus 1] > \gamma'/2$ (implying that $\gamma'/2 < \epsilon$) then we first observe that the set of $*$-inputs of $\chi_{\gamma'/2}$ strictly contains the set of $*$-inputs of $\chi_{\epsilon}$. Indeed, if $\chi_\epsilon = *$ then $\chi_{\gamma'/2} = *$ since $\gamma'/2 < \epsilon$, and moreover $\chi_\epsilon(j_0) \neq *$ whereas $\chi_{\gamma'/2}(j_0) = *$.

This allows us to apply the inductive hypothesis to $P := P$, $\mu := \mu$, and $\epsilon := \gamma'/2$. We obtain $d'',\tau'',\gamma''$ such that if $d \geq d''$, $\tau \leq \tau''$, $\gamma \leq \gamma''$ then the assumptions imply that for every $\alpha$ consistent with $\chi_{\gamma'/2}$ (in particular, every $\alpha$ consistent with $\chi_\alpha$) we have $\alpha \in P$ and
\[
 \Pr_{(y^{(1)},\dots,y^{(m)}) \sim \mu^n}[(\phi_1(y^{(1)}),\dots,\phi_m(y^{(m)})) = \alpha] > \gamma''.
\]

Choosing $d = \max(d',d'')$, $\tau = \min(\tau',\tau'')$, $\gamma = \min(\gamma'/2,\gamma'')$ completes the proof.
\end{proof}

\subsection{Rounding} \label{sec:main-rounding}

The proof of \Cref{thm:intro-main} (for predicates without affine relations) follows the main plan of the proof of \Cref{thm:intro-monotone}: we apply Jones' regularity lemma on $f_1,\dots,f_m$ to construct a junta $J$ such that on average, $f_j|_{J \gets x^{(j)}}$ is regular. \Cref{lem:counting} suggests a ``rounding'' procedure to construct the generalized polymorphism $(g_1,\dots,g_m)$: we define $g_j|_{J \gets x^{(j)}} \equiv b$ if $\Pr[f_j|_{J \gets x^{(j)}} = b] \geq 1 - \epsilon$ for some $b \in \{0,1\}$, and $g_j|_{J \gets x^{(j)}} = f_j|_{J \gets x^{(j)}}$ otherwise. The lemma implies (for an appropriate choice of parameters) that $(g_1|_{J \gets x^{(1)}}, \dots, g_m|_{J \gets x^{(m)}})$ is a generalized polymorphism whenever all of the involved subfunctions are regular, but gives no information otherwise.

In the monotone setting, we handle this difficulty by setting $g_j|_{J \gets x^{(j)}} \equiv 0$ whenever $f_j|_{J \gets x^{(j)}}$ is not regular. In the general case, there is no such safe direction. Instead, we use a two-step sampling procedure, along the lines of the proof of \Cref{lem:monotone-counting-NAND} (the counting lemma for NAND), to define $g_1,\dots,g_m$, and use \Cref{lem:counting} to show that $g_j$ is close to $f_j$.

The two-step sampling procedures in \Cref{lem:monotone-counting-NAND} and \Cref{lem:counting} rely on the existence of inputs $w_{(j)} \in P$ such that $w_{(j)}^{\oplus j} \in P$. If such inputs exist for all $j$ then the proof of \Cref{thm:intro-main} can be simplified considerably; we take this route when we prove \Cref{thm:intro-alphabet}. However, in general this is not the case. As an example, if $P = \{(1,0,0),(0,1,0),(0,0,1)\}$ then no such inputs exist for any~$j$. Our argument will need to distinguish between these two types of coordinates.

\begin{definition}[Flexible coordinates]
Let $P \subseteq \{0,1\}^m$ be a non-empty predicate. A coordinate $j \in [m]$ is \emph{flexible} if there exists $w_{(j)} \in P$ such that $w_{(j)}^{\oplus j} \in P$, and \emph{inflexible} otherwise.

If $j$ is a flexible coordinate, let $w_{(j,0)} \in P$ be some fixed input such that $(w_{(j,0)})_j = 0$ and $w_{(j,1)} := w_{(j,0)}^{\oplus j} \in P$. Also, let $w_{(j,*)}$ be the restriction obtained by changing the $j$'th coordinate to $*$.
\end{definition}

We can now describe the distribution of the restriction $\rho$ in the first step of the two-step sampling process. Each of its coordinates will be sampled according to the following distribution $\nu$ on $P \cup \{ w_{(j,*)} : j \text{ flexible} \}$, for a small enough $q > 0$:
\begin{itemize}
\item For $w \in P$, sample $w$ with probability
\[
 \mu(w) - \sum_{j\colon w = w_{(j,0)}} (1-p_j)q - \sum_{j\colon w = w_{(j,1)}} p_j q,
\]
where the sum is over all flexible coordinates $j$, and $p_j = \Pr[\mu|_j = 1]$.
\item For every flexible $j$, sample $w_{(j,*)}$ with probability $q$.
\end{itemize}
We can sample $\mu$ by first sampling $\nu$, and in case $w_{(j,*)}$ was sampled, sampling the $j$'th coordinate according to $\mu|_j$. Also, the distribution of $\nu|_j$ conditioned on $\nu|_j \neq *$ coincides with $\mu|_j$.

We start by explaining how to use the two-step sampling process to construct a generalized polymorphism $(g_1,\dots,g_m)$. In the lemma below, $J$ is the set which will be constructed using Jones' regularity lemma. \highlight{This rounding procedure prevents us from ensuring that $g_i = g_j$ even if $f_i = f_j$ and $\mu|_i = \mu|_j$.} In this lemma and below, we use $\min(\mu) := \min_{w \in P} \mu(w)$.

\begin{lemma}[Rounding lemma] \label{lem:rounding-poly}
Let $P \subseteq \{0,1\}^m$ be a non-empty predicate, and let $\mu$ be a distribution on $P$ having full support.
For every $\epsilon,\zeta > 0$ there exists $\delta > 0$ such that the following holds.

Let $f_1,\dots,f_m\colon \{0,1\}^n \to \{0,1\}$, and let $J \subseteq [n]$. 
For a restriction $\rho\colon J^c \to \{0,1,*\}^m$, define functions $g_1^{\rho,\epsilon},\dots,g_m^{\rho,\epsilon}\colon \{0,1\}^n \to \{0,1\}$ as follows: for each $j \in [m]$ and each $x^{(j)} \in \{0,1\}^J$,
\[
 g^{\rho,\epsilon}_j|_{J \gets x^{(j)}} =
 \begin{cases}
     0 & \text{if } \Ex|_{\mu|_j}[f_j|_{J \gets x^{(j)},\rho|_j}] \leq \epsilon, \\
     1 & \text{if } \Ex|_{\mu|_j}[f_j|_{J \gets x^{(j)},\rho|_j}] \geq 1-\epsilon, \\
     f_j|_{J \gets x^{(j)}} & \text{otherwise}.
 \end{cases}
\]
For fixed $x^{(1)},\dots,x^{(m)}$, this corresponds to a coloring $\chi^{\rho,\epsilon}\colon [m] \to \{0,1,*\}$ in a natural way.

If $(f_1,\dots,f_m)$ is a $(\mu,\min(\mu)^{|J|} \delta)$-approximate generalized polymorphism of $P$ then
\[
 \Pr_{\rho \sim \nu^{J^c}}[(g_1^{\rho,\epsilon},\dots,g_m^{\rho,\epsilon}) \text{ is a generalized polymorphism of $P$}] \geq 1-\zeta.
\]
\end{lemma}

If $j$ is inflexible then $\rho|_j \in \{0,1\}^{J^c}$, and so $f_j|_{J \gets x^{(j)}, \rho|_j}$ is a constant. In this case $g_j^{\rho,\epsilon}|_{J \gets x^{(j)}}$ always gets set to a constant. As we show later, \Cref{lem:counting} explains why this is a sound choice.

\begin{proof}
Suppose that $(f_1,\dots,f_m)$ is a $(\mu,\min(\mu)^{|J|} \delta)$-approximate generalized polymorphism of $P$, where we set $\delta$ later on. This means that
\[
 \Ex_{\rho \sim \nu^{J^c}} \Pr_{\substack{(x^{(1)},\dots,x^{(m)}) \sim \mu^J \\ (y^{(1)},\dots,y^{(m)}) \sim \mu^{J^c}|\rho}}[(f_1|_{J \gets x^{(1)}}(y^{(1)}), \dots, f_m|_{J \gets x^{(m)}}(y^{(m)})) \notin P] \leq \min(\mu)^{|J|} \delta,
\]
and so with probability at least $1-\zeta$ over the choice of $\rho$,
\[
 \Pr_{\substack{(x^{(1)},\dots,x^{(m)}) \sim \mu^J \\ (y^{(1)},\dots,y^{(m)}) \sim \mu^{J^c}|\rho}}[(f_1|_{J \gets x^{(1)}}(y^{(1)}), \dots, f_m|_{J \gets x^{(m)}}(y^{(m)})) \notin P] \leq \min(\mu)^{|J|}\delta/\zeta.
\]
Therefore for every $(x^{(1)},\dots,x^{(m)}) \in P^J$ we have
\[
 \Pr_{(y^{(1)},\dots,y^{(m)}) \sim \mu^{J^c}|\rho}[(f_1|_{J \gets x^{(1)}}(y^{(1)}), \dots, f_m|_{J \gets x^{(m)}}(y^{(m)})) \notin P] \leq \delta/\zeta.
\]
We will show that if $\delta$ is small enough, then this implies that $(g_1^\rho,\dots,g_m^\rho)$ is a generalized polymorphism.

If $(g_1^{\rho,\epsilon},\dots,g_m^{\rho,\epsilon})$ is not a generalized polymorphism of $P$ then there exists $(x^{(1)},\dots,x^{(m)}) \in P^J$ such that $(g_1^{\rho,\epsilon}|_{J \gets x^{(1)}},\dots,g_m^{\rho,\epsilon}|_{J \gets x^{(m)}})$ is not a generalized polymorphism of $P$. This means that the corresponding coloring $\chi^{\rho,\epsilon}$ can be extended to some $\alpha \notin P$. Observe that
\[
 \Pr_{(y^{(1)},\dots,y^{(m)}) \sim \mu^{J^c}|\rho}[(f_1|_{J \gets x^{(1)}}(y^{(1)}), \dots, f_m|_{J \gets x^{(m)}}(y^{(m)})) = \alpha] \geq \epsilon^m.
\]
Choosing $\delta = \epsilon^m\zeta/2$ completes the proof.
\end{proof}

The next step is to show that on average (over $\rho$), the function $g_j^{\rho,\eta}$ is close to $f_j$, for an appropriate choice of $\eta$. For flexible coordinates, this follows immediately from \Cref{thm:it-aint-over-two-sided}. For inflexible coordinates, we will appeal to \Cref{lem:counting}, which implies that for such coordinates, $f_j|_{J \gets x^{(j)}}$ is close to constant. Since we apply \Cref{lem:counting}, this argument will only work for certain subfunctions.

\begin{definition}[Good subfunctions] \label{def:good}
Let $P \subseteq \{0,1\}^m$ be a predicate, and let $\mu$ be a distribution over $P$ having full support. Let $d \in \mathbb{N}$ and $\tau > 0$. Let $f_1,\dots,f_m\colon \{0,1\}^n \to \{0,1\}$, and let $J \subseteq [n]$.

A subfunction $f_j|_{J \gets x}$ is \emph{$(\mu,d,\tau)$-good} if there exists $(x^{(1)},\dots,x^{(m)}) \in P^J$ with $x^{(j)} = x$ such that $f_k|_{J \gets x^{(k)}}$ is $(d,\tau)$-regular with respect to $\mu|_k$ for all $k \in [m]$.
\end{definition}

\begin{lemma}[Most subfunctions are good]
\label{lem:good-most}
Let $P \subseteq \{0,1\}^m$ be a predicate, and let $\mu$ be a distribution over $P$ having full support. Let $d \in \mathbb{N}$ and $\tau > 0$. Let $f_1,\dots,f_m\colon \{0,1\}^n \to \{0,1\}$, and let $J \subseteq [n]$.

Suppose that for all $j$,
\[
 \Pr_{x \sim \mu|_j^J}[f_j|_{J \gets x} \text{ is $(d,\tau)$-regular wrt $\mu|_j$}] \geq 1-\epsilon.
\]
Then for all $j$,
\[
 \Pr_{x \sim \mu|_j^J}[f_j|_{J \gets x} \text{ is $(\mu,d,\tau)$-good}] \geq 1-m\epsilon.
\]
\end{lemma}
\begin{proof}
Fix $j \in [m]$. The assumption implies that
\[
 \Pr_{(x^{(1)},\dots,x^{(m)}) \sim \mu^J}[f_k|_{J \gets x^{(k)}} \text{ is $(d,\tau)$-regular wrt $\mu|_k$ for all $k$}] \geq 1-m\epsilon.
\]
If the event happens then $f_j|_{J \gets x^{(j)}}$ is $(\mu,d,\tau)$-good. The lemma immediately follows.
\end{proof}

We can now show that on average, $g^{\rho,\eta}_j$ is close to $f_j$, for an appropriate choice of $\eta$. Recall that $\min(\mu) = \min_{w \in P} \mu(w)$.

\begin{lemma}[Soundness of rounding] \label{lem:rounding-soundness}
Let $P \subseteq \{0,1\}^m$ be a non-empty predicate without affine relations and let $\mu$ be a distribution on $P$ having full support. For every $\epsilon > 0$ there exist $d \in \mathbb{N}$ and $\delta, \eta, \tau > 0$ such that the following holds.

Let $f_1,\dots,f_m\colon \{0,1\}^m \to \{0,1\}$, and let $J \subseteq [n]$. Suppose that for all $j \in [m]$,
\[
 \Pr_{x^{(j)} \sim \mu|_j^J}[f_j|_{J \gets x^{(j)}} \text{ is $(d,\tau)$-regular}] \geq 1 - \epsilon.
\]
If $(f_1,\dots,f_m)$ is a $(\mu,\min(\mu)^{|J|} \delta)$-approximate generalized polymorphism of $P$ then for all $j \in [m]$,
\[
 \Ex_{\rho \sim \nu^J} \Pr_{\mu|_j}[g_j^{\rho,\eta} \neq f_j] = O(\epsilon).
\]
\end{lemma}
\begin{proof}
Apply \Cref{lem:counting} with $\epsilon := \epsilon$ to obtain $d', \tau', \gamma'$ such that the assumptions imply the following, assuming $d \geq d'$, $\tau \leq \tau'$, $\delta \leq \gamma'$. For all $(x^{(1)},\dots,x^{(m)}) \in P^J$, if $f_1|_{J \gets x^{(1)}},\dots,f|_{J \gets x^{(m)}}$ are all $(d,\tau)$-regular then all assignments extending $\chi_\epsilon$ belong to $P$.

For every flexible $j$, apply \Cref{thm:it-aint-over} (the one-sided version of \Cref{thm:it-aint-over-two-sided}) with $p := \Pr[\mu|_j = 1]$, $q := q$ (the parameter used to define $\nu$) and $\epsilon := \epsilon$ to obtain $d_j,\tau_j,\delta_j$ such that for $b \in \{0,1\}$ and $x^{(j)} \in \{0,1\}^J$, if $\chi_\epsilon(j) \neq b$  and $\eta \leq \delta_j$ then
\[
 \Pr_{\rho \sim \nu^J}[\chi^{\rho,\eta}(j) \neq b] \geq 1-\epsilon.
\]
(Note that $\chi_\epsilon(j) \neq b$ iff $\Pr_{\mu|_j}[f_j|_{J \gets x^{(j)}} = b \oplus 1] > \epsilon$, and $\chi^{\rho,\eta}(j) \neq b$ iff $\Pr_{\mu_j}[f_j|_{J \gets x^{(j)},\rho|_j} = b \oplus 1] > \eta$.)

Let $F$ be the set of flexible coordinates. We take $d = \max(d',(d_j)_{j \in F})$, $\tau = \min(\tau',(\tau_j)_{j \in F})$, $\delta = \gamma'$, $\eta = \min(\epsilon, (\delta_j)_{j \in F})$.

Given $j$, observe that
\begin{multline*}
 \Ex_{\rho \sim \nu^J}[\Pr_{\mu|_j}[g_j^{\rho,\eta} \neq f_j] ]\leq \Pr_{x^{(j)} \sim \mu|_j^J}[f_j|_{J \gets x^{(j)}} \text{ is not $(\mu,d,\tau)$-good}] + \\ \sum_{x^{(j)}\colon f_j|_{J \gets x^{(j)}} \text{ is  $(\mu,d,\tau)$-good}} \mu|_j(x^{(j)}) \Ex_{\rho \sim \nu^J} \Pr_{\mu|_j}[g_j^{\rho,\eta}|_{J \gets x^{(j)}} \neq f_j|_{J \gets x^{(j)}}].
\end{multline*}
The first summand is at most $m\epsilon$, and so it suffices to show that for all $x^{(j)}$ such that $f_j|_{J \gets x^{(j)}}$ is $(\mu,d,\tau)$-good,
\[
 \Ex_{\rho \sim \nu^J} \Pr_{\mu|_j}[g_j^{\rho,\eta}|_{J \gets x^{(j)}} \neq f_j|_{J \gets x^{(j)}}] = O(\epsilon).
\]

Unpacking the definition of $g_j^{\rho,\eta}$, the left-hand side is
\[
 \Pr_{\rho \sim \nu|_j^J}[\Ex[f_j|_{J \gets x^{(j)},\rho}] \leq \eta] \cdot \Ex_{\mu|_j}[f_j|_{J \gets x^{(j)}}] +
 \Pr_{\rho \sim \nu|_j^J}[\Ex[1-f_j|_{J \gets x^{(j)},\rho}] \leq \eta] \cdot \Ex_{\mu|_j}[1-f_j|_{J \gets x^{(j)}}].
\]
The two summands are similar, so it suffices to bound the first one.

We consider two cases, according to whether $j$ is flexible or not. If $j$ is flexible then either $\chi_\epsilon(j) = 0$, in which case $\Ex[f_j|_{J \gets x^{(j)}}] \leq \epsilon$, or $\chi_\epsilon(j) \neq 0$, in which case $\Pr_\rho[\chi^{\rho,\eta} = 0] \leq \epsilon$, that is, $\Pr_{\rho \sim \nu|_j^J}[\Ex[f_j|_{J \gets x^{(j)},\rho}] \leq \eta] \leq \epsilon$. In both cases, the summand is bounded by $\epsilon$.

If $j$ is inflexible then we need to use the fact that every extension of $\chi_\epsilon$ belongs to $P$. Since $j$ is inflexible, this implies that $\chi_\epsilon(j) \neq *$. We consider two subcases, according to the value of $\chi_\epsilon(j)$.
If $\chi_\epsilon(j) = 0$ then $\Ex[f_j|_{J \gets x^{(j)}}] \leq \epsilon$. If $\chi_\epsilon(j) = 1$ then $\Ex[f_j|_{J \gets x^{(j)}}] \geq 1 - \epsilon$. Since $j$ is inflexible, $\nu|_j = \mu|_j$, and so $f_j|_{J \gets x^{(j)}, \rho}$ is a constant, which equals~$0$ with probability at most $\epsilon$. In both cases, the summand is bounded by $\epsilon$.
\end{proof}

\subsection{Proof without short affine relations} \label{sec:main-generic}

If $P$ has no affine relations then we can complete the proof as in \Cref{sec:monotone} as follows. First, we apply Jones' regularity lemma on $f_1,\dots,f_m$ to obtain $J$. We then combine \Cref{lem:rounding-poly} and \Cref{lem:rounding-soundness} to find a restriction $\rho$ such that $(g_1^{\rho,\eta},\dots,g_m^{\rho,\eta})$ is a generalized polymorphism and $g_j^{\rho,\eta}$ is close to $f_j$ for all $j$.
In this section, we show how to modify this argument to handle linear relations of size at least $3$; handling smaller linear relations is easier and will be done in the next section.

For the remainder of the section, we assume that $P$ has no \emph{small} affine relations, meaning no affine relations of length smaller than~$3$.

The first step is to peel off all affine relations. Below we use the notation
\[
 \chi_{S,b}(x) = b \oplus \bigoplus_{i \in S} x_i.
\]

\begin{lemma}[Peeling affine relations] \label{lem:peeling}
Let $P$ be a predicate without small affine relations, and let $\mu$ be a distribution on $P$ with full support. There exist $\epsilon_0 > 0$ and sets $F \subseteq I \subseteq [m]$ such that $P|_I$ has no affine relations, and the following holds for all $\epsilon \leq \epsilon_0$.

Let $f_1,\dots,f_m\colon \{0,1\}^n \to \{0,1\}$ be a $(\mu,\epsilon)$-approximate generalized polymorphism of $P$.
\begin{enumerate}[(a)]
\item For each $j \notin F$ there exist $b_j \in \{0,1\}$ and $S_j \subseteq [n]$ such that $\Pr_{\mu|_j}[f_j \neq \chi_{S_j,b_j}] = O(\epsilon)$.
\item If $(g_j)_{j \in I}$ is a generalized polymorphism of $P|_I$ such that $g_j = \chi_{S_j,b_j}$ for all $j \in I \setminus F$, then we can extend it to a generalized polymorphism of $P$ by taking $g_j = \chi_{S_j,b_j}$ for $j \notin I$.
\end{enumerate}
\end{lemma}

The proof uses \Cref{thm:intro-blr}, proved in \Cref{sec:linearity}, which is the special case of \Cref{thm:intro-main} for affine relations. 

\begin{proof}
We construct $F,I$ using an iterative process. During the process, each coordinate can be active or not active; originally all are active. Furthermore, each coordinate has a list of characters $\chi_{S,b}$, initially empty. The process ends once there is no affine relation involving only active coordinates.

Each step of the iteration proceeds as follows. Choose an affine relation involving the coordinates in some set $A$, all of them active. Since $P$ has no small affine relations, $|A| \geq 3$. Therefore we can apply \Cref{thm:intro-blr} to obtain, for each $j \in A$, a character $\chi_{S_j,b_j}$ such that $\Pr_{\mu|_j}[f_j \neq \chi_{S_j,b_j}] = O(\epsilon)$.
For each $j \in A$, we add the character $\chi_{S_j,b_j}$ to the list of characters for coordinate $j$. We also pick a coordinate $j_0 \in A$ arbitrarily, and render it inactive.

After the process ends, some of the lists are empty, and they comprise the set $F$. The set $I$ contains all active coordinates. Property~(a) is automatically satisfied.

The list of characters for each $j \notin F$ could contain more than one character. We can rule this out by taking $\epsilon_0$ small enough. Indeed, if the list for $j$ contains two different characters $\chi_{S',b'},\chi_{S'',b''}$ then
\[
 \Pr_{\mu|_j}[\chi_{S',b'} \neq \chi_{S'',b''}] = O(\epsilon).
\]
If $\epsilon_0$ is smaller than a constant, this rules out $S' = S''$.
Take any $i \in S' \triangle S''$, and sample coordinate $i$ last. Whether $\chi_{S',b'} \neq \chi_{S'',b''}$ or not depends on the value of coordinate $i$, and so the probability that $\chi_{S',b'} \neq \chi_{S'',b''}$ is at least $\min(\mu|_j(0), \mu|_j(1))$. Therefore, choosing $\epsilon_0 = c \min_{j,b} \mu|_j(b)$ for an appropriate $c > 0$ ensures that every non-empty list contains precisely one character.

It remains to prove Property~(b). For this, we observe that if $w \in P|_I$ then there is a unique way to extend it to an element in $P$. This is precisely how we extend the generalized polymorphism in Property~(b).
\end{proof}

We would like to apply the argument outlined in the beginning of the section to $P|_I$ in such a way that guarantees that $g^{\rho,\eta}_j = \chi_{S_j,b_j}$ for all $j \in I \setminus F$. We do this in two steps. First, we change $f_j$ to $f'_j = \chi_{S_j,b_j}$ for all $j \in I \setminus F$, which increases the error probability in a controlled way. Second, we ensure somehow that $g^{\rho,\eta}_j = f'_j$ for all $j \in I \setminus F$. This can be done in two ways: either $J$ contains $S_j$, or $S_j \setminus J$ is so large that $\chi^{\rho,\eta} = *$ is very likely, as given by the following lemma.

\begin{lemma}[Large characters are balanced] \label{lem:large-characters}
Let $\mu$ be a full support distribution on $\{0,1\}$. For every $q \in (0,1)$ and $\eta,\zeta > 0$ there exists $M \in \mathbb{N}$ such that the following holds.

Let $\nu$ be the distribution on $\{0,1,*\}$ given by $\nu(*) = q$ and $\nu(b) = (1-q)\mu(b)$.
If $f = \chi_{S,b}$ for $|S| \geq M$ and $b \in \{0,1\}$ then
\[
 \Pr_{\rho \sim \nu}[\eta < \Ex_{\mu}[f|_\rho] < 1-\eta] \geq 1 - \zeta.
\]
\end{lemma}
\begin{proof}
The proof is straightforward, so we only outline it. Since $\mu$ has full support, we can find $M' \in \mathbb{N}$ such that if $f' = \chi_{S',b'}$ for $|S'| \geq M'$ then $\eta < \Ex_{\mu}[f'] < 1-\eta$. We can find $M$ for which $f|_\rho$ is of this form with probability $1 - \zeta$.
\end{proof}

We now put everything together.

\begin{proof}[Proof of \Cref{thm:intro-main} for predicates without short affine relations]
Apply \Cref{lem:peeling} to obtain $F \subseteq I \subseteq [m]$ and $\epsilon_0 > 0$.
Recall that the lemma states that $P|_I$ has no affine relations, and that the following holds for every $(\mu,\delta)$-approximate polymorphism of $P$, whenever $\delta \leq \epsilon_0$.

First, for each $j \notin F$, the function $f_j$ is $O(\delta)$-close to some $\chi_{S_j,b_j}$. Second, any generalized polymorphism $(g_j)_{j \in I}$ of $P|_I$ which satisfies $g_j = \chi_{S_j,b_j}$ for $j \in I \setminus F$ can be extended to a generalized polymorphism of $P$ by taking $g_j = \chi_{S_j,b_j}$ for $j \notin I$. 

We can assume without loss of generality that $\epsilon \leq \epsilon_0$ (otherwise replace $\epsilon$ with $\epsilon_0$).
We will prove the theorem with an error probability of $O(\epsilon)$ rather than $\epsilon$ for convenience.

\paragraph{Constructing the junta}

The junta is constructed by applying various lemmas proved in this section. In order to make the argument more readable, we briefly recall the statement of each lemma. We use the following notation: a \emph{random $J$-subfunction of $f'_j$} is obtained from $f'_j$ by restricting the coordinates in $J$ according to $\mu|_j$.
Also, $\rho$ is a random restriction of the coordinates in $J^c$ sampled according to the distribution $\nu$ defined in \Cref{sec:main-rounding}.

Apply \Cref{lem:rounding-soundness} (Soundness of rounding) with the predicate $P|_I$, the distribution $\mu|_I$, and $\epsilon := \epsilon$ to obtain $d',\delta',\eta',\tau'$ such that the following holds. Given a set $J$, suppose that for each $j \in I$, a random $J$-subfunction of $f'_j$ is $(d',\tau')$-regular w.p.\ $1 - \epsilon$; and that $(f_1,\dots,f_m)$ is a $(\mu,\min(\mu)^{|J|} \delta')$-approximate polymorphism of $P|_I$. Then for all $j$, $\Ex_\rho \Pr[g_j^{\rho,\eta'} \neq f'_j] = O(\epsilon)$.

Apply \Cref{thm:jones-junta} (Jones' regularity lemma) with $\epsilon := \epsilon$, $d = d'$, $\tau = \tau'$, and $p_j = \Pr[\mu|_j = 1]$ for all $j \in F$ to obtain $\mathcal{M}\colon \mathbb{N} \to \mathbb{N}$ such that the following holds. For each set $J_0$ there exists a set $J \subseteq J_0$ of size at most $\mathcal{M}(|J_0|)$ such that for all $j \in F$, a random $J$-subfunction of $f'_j$ is $(d',\tau')$-regular w.p.\ $1-\epsilon$.

For every $j \in I \setminus F$, apply \Cref{lem:large-characters} with $\mu := \mu|_j$, $\eta := \eta'$, $\zeta := 1/(2m)$, and the $q$ in the definition of $\nu$ (\Cref{sec:main-rounding}) to obtain $M_j$. The lemma implies that if $|S_j \setminus J| \geq M_j$ then for every $x^{(j)} \in \{0,1\}^J$ we have $\eta' < \Ex[\chi_{S_j,b_j}|_{J \gets x^{(j)}, \rho|_j}] < 1 - \eta'$ with probability $1-1/(2m)$ over the choice of $\rho$.

The set $J_0$ that we construct will be based on an application of \Cref{lem:peeling}, which will yield us a set $S_j$ for each $j \notin F$. We would like the set $J$ constructed by Jones' regularity lemma to satisfy the following, for each $j \in I \setminus F$: either $J$ contains $S_j$, or $|S_j \setminus J| \geq M_j$.
Later on, we determine a value of $M$, depending only on the function $\mathcal{M}$ and the parameters $M_j$ for $j \in I \setminus F$, such that such a set $J$ can be constructed of size at most $M$.

Apply \Cref{lem:rounding-poly} (Rounding lemma) with the predicate $P|_I$, the distribution $\mu|_I$, $\epsilon := \eta'$, and $\zeta := 1/(2m)$ to obtain $\delta''$ such that the following holds. If $(f'_1,\dots,f'_m)$ is a $(\mu,\min(\mu)^{|J|} \delta'')$-approximate polymorphism of $P|_I$ then $(g_j^{\rho,\eta'})_{j \in I}$ is a generalized polymorphism of $P|_I$ w.p.\ $1-1/(2m)$ over the choice of $\rho$.

In order to apply \Cref{lem:rounding-soundness,lem:rounding-poly}, we will need the error probability to be at most $\min(\mu)^{|J|} \min(\delta',\delta'')$. Consequently we choose
\[
 \delta = c \min(\mu)^M \min(\delta',\delta'',\epsilon)/(m+1),
\]
for an appropriate constant $c \in (0, 1)$; we will see where the additional $m+1$ factor comes from in a moment. (Recall that $\delta$ is the parameter such that $(f_1,\dots,f_m)$ is a $(\mu,\delta)$-approximate polymorphism of $P$.)

Apply \Cref{lem:peeling} to $f_1,\dots,f_m$ to obtain, for each $j \notin F$, a function $g_j = \chi_{S_j,b_j}$ such that $\Pr_{\mu|_j}[g_j \neq f_j] \leq \min(\mu)^M \min(\delta',\delta'',\epsilon)/(m+1)$ (this is where we get $c$ from). Define
\[
 f'_j = \begin{cases}
     g_j & \text{if } j \notin F, \\
     f_j & \text{if } j \in F.
 \end{cases}
\]

Observe that $(f'_j)_{j \in I}$ is a $(\mu|_I,\min(\mu)^M \min(\delta',\delta'',\epsilon))$-approximate generalized polymorphism of $P|_I$, summing up the probability that $(f_j)_{j \in I} \notin P|_I$ with the probabilities that $f_j$ and $f'_j$ differ.

\medskip

We can now construct $J$.
First, we construct $J_0$. Start with $J_0 := \emptyset$, and while some $j \in I \setminus F$ satisfies $|S_j \setminus J_0| < M_j + \mathcal{M}(|J_0|)$, add $S_j$ to $J_0$. The process terminates with a set $J_0$ such that for each $j \in I \setminus F$, either $J_0 \supseteq S_j$, or $|S_j \setminus J_0| \geq M_j + \mathcal{M}(|J_0|)$.

We proceed to bound the size of $J_0$. Let $M' = \max_{j \in I \setminus F} M_j$, and assume without loss of generality that $\mathcal{M}$ is monotone (otherwise, replace it with $\mathcal{M}'(s) := \max_{t \leq s} \mathcal{M}(t)$). Taking $B_0 = 0$ and $B_{t+1} = B_t + M' + \mathcal{M}(B_t)$, a simple induction shows that after adding $t$ many $S_j$'s, we have $|J_0| \leq B_t$. Thus the final size of $J_0$ is at most $B_{|I \setminus F|}$.

Finally, we obtain $J$ by applying Jones' regularity lemma to the functions $f'_j$ for $j \in F$. Observe that $|J| \leq M := \mathcal{M}(B_{|I \setminus F|})$. Furthermore, for each $j \in I \setminus F$, either $J_0 \supseteq S_j$, in which case $J \supseteq S_j$, or $|S_j \setminus J_0| \geq M_j + \mathcal{M}(|J_0|) \geq M_j + |J|$, in which case $|S_j \setminus J| \geq M_j$.

\paragraph{Finding a good restriction}

We would like to find a restriction $\rho$ such that all of the following hold:
\begin{enumerate}[(a)]
\item $(g^{\rho,\eta}_j)_{j \in I}$ is a generalized polymorphism of $P|_I$.
\item For each $j \in I \setminus F$, $g^{\rho,\eta}_j = f'_j$.
\item For each $j \in F$, $\Pr_{\mu|_j}[g^{\rho,\eta}_j \neq f'_j] = O(\epsilon)$.
\end{enumerate}
Given such $\rho$, we define $g_j = g^{\rho,\eta}_j$ for $j \in I$ to complete the proof via \Cref{lem:peeling}.

\Cref{lem:rounding-poly} implies that the first property holds with probability $1-1/(2m)$. 
We go on to the second property. Let $j \in I \setminus F$. If $S_j \subseteq J$ then $g^{\rho,\eta}_j = f'_j$ always. Otherwise, $|S_j \setminus J| \geq M_j$, and so \Cref{lem:large-characters} implies that $g^{\rho,\eta}_j = f'_j$ with probability at least $1 - 1/(2m)$ (crucially, the property $\chi_j^{\rho,\eta} = *$ doesn't depend on the input $x^{(j)}$ to $J$). In total, the first two properties hold with probability at least $1/2$.

\Cref{lem:rounding-soundness} implies that
\[
 \Ex_\rho \sum_{j \in F} \Pr_{\mu|_j}[g^{\rho,\eta}_j \neq f'_j] = O(\epsilon).
\]
Conditioning on the first two properties, this still holds (with the right-hand side doubled). Hence there exists a restriction $\rho$ satisfying all three properties, completing the proof.
\end{proof}

\subsection{Concluding the proof} \label{sec:main-conclusion}

In this short section, we complete the proof of \Cref{thm:intro-main} by handling short affine relations.

\begin{proof}[Proof of \Cref{thm:intro-main}]
We prove the theorem with an error term of $O(\epsilon)$ rather than $\epsilon$, for simplicity.

By negating coordinates, we can assume that all short affine relations are of the following forms: (i) $w_j = 0$ for all $w \in P$, (ii) $w_j = w_k$ for all $w \in P$.

Let $Z$ be the set of coordinates in $P$ that are always~$0$. Let $I$ consist of a choice of one coordinate from each set of equivalent non-constant coordinates. For each $j \in I$, let $I_j$ be the coordinates equivalent to $j$; possibly $I_j = \{j\}$. The sets $Z,(I_j)_{j \in I}$ thus partition $[m]$.

The predicate $P|_I$ has no short affine relations, and so the special case handled in \Cref{sec:main-generic} applies to it, giving us a value of $\delta$. We will prove the theorem for $\delta := \min(\delta, \epsilon, 1)$.

Let $(f_1,\dots,f_m)$ be a $(\mu,\delta)$-approximate generalized polymorphism of $P$. Applying the special case, we obtain a generalized polymorphism $(g_j)|_{j \in I}$ of $P|_I$. We complete it to a generalized polymorphism of $P$ by taking $g_j = 0$ for $j \in Z$ and $g_k = g_j$ for $k \in I_j$ (where $j \in I$).

It remains to bound the distance between $f_j$ and $g_j$ for $j \notin I$. If $j \in Z$ then $f_j(0,\dots,0) = 0$ since $(f_1,\dots,f_m)$ is an approximate polymorphism, and so $\Pr_{\mu|_j}[g_j \neq f_j] = 0$. If $k \in I_j$ then $\Pr_{\mu|_j}[f_k \neq f_j] \leq \delta$ and so $\Pr_{\mu|_j}[g_k \neq f_k] \leq \epsilon+\delta \leq 2\epsilon$. This completes the proof.
\end{proof}

\section{Linearity testing} \label{sec:linearity}

In this section we prove \Cref{thm:intro-blr}, which extends the BLR test to arbitrary distributions. The proof uses ideas from~\cite{DFH2025}, which extended a related test to a specific family of distributions.

\thmintroblr*

Before describing the proof, we observe when $\epsilon$ is small enough (an assumption we can make without loss of generality), we can guarantee the ``furthermore'' clause. Indeed, if $f_i = f_j$ and $\mu|_i = \mu|_j$ then
\[
 \Pr_{\mu|_i]}[g_i \neq g_j] \leq 2\epsilon.
\]
As shown during the course of the proof of \Cref{lem:large-characters}, for small enough $\epsilon$ (as a function only of $\mu|_i$) this implies that $g_i = g_j$.

The starting point of the proof is the case in which $\mu$ is the uniform distribution over $P_{m,b}$, which we denote by $\pi$. This case is standard, but for completeness we prove it in \Cref{sec:linearity-uniform}.

The general case is handled by reduction to the uniform distribution, using an agreement theorem from~\cite{DFH2025}. The basic idea is to sample $x \sim \mu$ in two steps. In the first step, we sample $x \sim \nu$ with probability $1-q$ (for appropriate $\nu,q$), and leave it unsampled with probability $q$. In the second case, if $x$ wasn't sampled, we sample it according to $\pi$.

If we stop after the first step, we are in a position to apply the result of \Cref{sec:linearity-uniform}, deducing that the resulting subfunctions are close to characters. We show that these characters agree with each other, and apply the agreement theorem to deduce that on average, they are restrictions of the same character (up to sign).

Our argument will require two agreement theorems: one from \cite{DFH2025}, and a folklore result about mixing Markov chains which we prove in \Cref{sec:mixing}.
The reduction itself is described in \Cref{sec:linearity-arbitrary}, where we also provide a more detailed outline.

\subsection{Uniform distribution}
\label{sec:linearity-uniform}

In this section we prove \Cref{thm:intro-blr} in the case of the uniform distribution. The proof follows the classical argument of~\cite{BCHKS96}.

\begin{proof}[Proof of \Cref{thm:intro-blr} for the uniform distribution]
Since we are going to use Fourier analysis, it will be more convenient to switch from $\{0,1\}$ to $\{-1,1\}$. Accordingly, we assume that $f_1,\dots,f_m\colon \{-1,1\}^n \to \{-1,1\}$ satisfy
\[
 \Pr_{x^{(1)},\dots,x^{(m)}}[f_1(x^{(1)}) \cdots f_m(x^{(m)}) = B] \geq 1 - \delta,
\]
where $B = (-1)^b$ and $x^{(1)},\dots,x^{(m)}$ are sampled as follows: the first $m-1$ are sampled uniformly and independently, and $x^{(m)}_i = B x^{(1)}_i \cdots x^{(m-1)}_i$.

The assumption is equivalent to
\[
 \Ex_{x^{(1)},\dots,x^{(m)}}[B f_1(x^{(1)}) \cdots f_m(x^{(m)})] \geq 1 - 2\delta.
\]
Substituting the Fourier expansions gives
\[
 1 - 2\delta \leq \sum_{S_1,\dots,S_m} B \hat{f}_1(S_1) \cdots \hat{f}_m(S_m) \Ex_{x^{(1)},\dots,x^{(m-1)}}[\chi_{S_1}(x^{(1)}) \cdot \chi_{S_{m-1}}(x^{(m-1)}) \chi_{S_m}(B x^{(1)} \cdots x^{(m-1)})],
\]
where $\chi_S(x) = \prod_{i \in S} x_i$. Multiplicativity and orthogonality of characters shows that the expectation vanishes unless all $S_j$ are equal, and so we obtain
\[
 1 - 2\delta \leq \sum_S B^{|S|+1} \hat{f}_1(S) \cdots \hat{f}_m(S) \leq \max_S |\hat{f}_1(S)| \sum_S |\hat{f}_2(S) \cdots \hat{f}_m(S)| \leq \max_S |\hat{f}_1(S)| \sum_S |\hat{f}_2(S) \hat{f}_3(S)|,
\]
where we used $m \geq 3$.

At this point, we apply the Cauchy--Schwarz inequality to bound the second factor on the right:
\[
 1 - 2\delta \leq \max_S |\hat{f}_1(S)| \sqrt{\sum_S \hat{f}_2(S)^2} \sqrt{\sum_S \hat{f}_3(S)^2} = \max_S |\hat{f}_1(S)|.
\]
Therefore there exists $S_1$ such that $|\hat{f}_1(S_1)| \geq 1 - 2\delta$. Let $B_1$ be the sign of $\hat{f}_1(S_1)$. Then
\[
 \Pr[f_1 = B_1 \chi_{S_1}] \geq 1-\delta.
\]
Accordingly, we choose $\delta = \epsilon$.

In exactly the same way, we can find $B_i \chi_{S_i}$ approximating $f_2,\dots,f_m$:
\[
 \Pr[f_i = B_i \chi_{S_i}] \geq 1-\delta.
\]
It follows that for $B' = B_1 \cdots B_m$,
\[
 \Pr_{x^{(1)},\dots,x^{(m)}}[f_1(x^{(1)}) \cdots f_m(x^{(m)}) = B' \chi_{S_1}(x^{(1)}) \cdots \chi_{S_m}(x^{(m)})] \geq 1 - m\delta,
\]
and so
\[
 \Pr_{x^{(1)},\dots,x^{(m)}}[B' \chi_{S_1}(x^{(1)}) \cdots \chi_{S_m}(x^{(m)}) = B] \geq 1 - (m+1)\delta.
\]

At this point we recall the definition of $x^{(m)}$, which implies that
\[
 B' \chi_{S_1}(x^{(1)}) \cdots \chi_{S_m}(x^{(m)}) = B'B^{|S_m|} \chi_{S_1 \triangle S_m}(x^{(1)}) \cdots \chi_{S_{m-1} \triangle S_m}(x^{(m-1)}).
\]
If not all $S_i$ are equal then
\[
 \Pr_{x^{(1)},\dots,x^{(m)}}[B' \chi_{S_1}(x^{(1)}) \cdots \chi_{S_m}(x^{(m)}) = B] = \frac{1}{2},
\]
which we can rule out by ensuring, without loss of generality, that $\epsilon = \delta < 1/(2(m+1))$. Thus all $S_i$ are equal, concluding the proof.
\end{proof}

\subsection{Agreement lemma for mixing Markov chains}
\label{sec:mixing}

During the proof of \Cref{thm:intro-blr} we will encounter the following situation. There is a function $f\colon X \to \Sigma$ and a way to sample $x,y \in X$ in a coupled fashion such that $(x,y)$ and $(y,x)$ have the same distribution. Given $\Pr[f(x) \neq f(y)] \leq \epsilon$, we would like to deduce that $\Pr[f(x) \neq \sigma] = O(\epsilon)$ for some $\sigma \in \Sigma$. 

We can describe the coupling as a symmetric bistochastic $X \times X$ matrix $M$ whose entries describe the coupling. We denote by $\lambda(M)$ the second largest eigenvalue of $M$. If $\lambda(M) < 1$ then $M$ has a unique stationary distribution $\mu(M)$.

\begin{lemma}[Agreement lemma for Markov chains] \label{lem:agreement-lambda}
Let $X,\Sigma$ be finite sets, and let $M$ be a symmetric bistochastic $X \times X$ matrix with $\lambda = \lambda(M) < 1$ and stationary distribution $\mu = \mu(M)$.

If a function $f\colon X \to \Sigma$ satisfies
\[
 \Pr_{(x,y) \sim M}[f(x) \neq f(y)] \leq \epsilon
\]
then there exists $\sigma \in \Sigma$ such that
\[
 \Pr_{x \sim \mu}[f(x) \neq \sigma] \leq \frac{\epsilon}{1-\lambda}.
\]
\end{lemma}
\begin{proof}
For $\sigma \in \Sigma$, let
\[
 f_\sigma(x) = \begin{cases}
     1 & \text{if } f(x) = \sigma, \\
     0 & \text{if } f(x) \neq \sigma.
 \end{cases}
\]

Let $g_\sigma(x) = f_\sigma(x) - \Ex_\mu[f_\sigma]$, so that $\Ex_\mu[g_\sigma] = 0$. The assumptions on $M$ imply that
\[
 \Ex_{(x,y) \sim M}[g_\sigma(x) g_\sigma(y)] \leq \lambda \Ex_{x \sim \mu}[g_\sigma(x)^2].
\]

We can relate the left-hand side to agreement for $f_\sigma$:
\begin{align*}
 \Pr_{(x,y) \sim M}[f_\sigma(x) \neq f_\sigma(y)] &=
 \Ex_{(x,y) \sim M}[(f_\sigma(x) - f_\sigma(y))^2] \\ &=
 \Ex_{(x,y) \sim M}[(g_\sigma(x) - g_\sigma(y))^2] = 
 2\Ex_{x \sim \mu}[g_\sigma(x)^2] - 2\Ex_{(x,y) \sim M}[g_\sigma(x) g_\sigma(y)].
\end{align*}
It follows that
\[
 \Pr_{(x,y) \sim M}[f_\sigma(x) \neq f_\sigma(y)] \geq 2\Ex_{x \sim \mu}[g_\sigma(x)^2] - 2\Ex_{(x,y) \sim M}[g_\sigma(x) g_\sigma(y)] \geq 2(1-\lambda)\Ex_{x \sim \mu}[g_\sigma(x)^2].
\]

In a completely analogous way, we obtain
\[
 \Pr_{x,y \sim \mu}[f_\sigma(x) \neq f_\sigma(y)] = 2\Ex_{x \sim \mu}[g_\sigma(x)^2] - 2\Ex_{x,y \sim \mu}[g_\sigma(x) g_\sigma(y)] = 2\Ex_{x \sim \mu}[g_\sigma(x)^2] \leq \frac{\Pr_{(x,y) \sim M}[f_\sigma(x) \neq f_\sigma(y)]}{1-\lambda}.
\]

We now circle back to $f$. Summing the above inequality over all $\sigma$, we obtain
\[
 \sum_{\sigma \in \Sigma} \Pr_{x,y \sim \mu}[f_\sigma(x) \neq f_\sigma(y)] \leq \frac{1}{1-\lambda} \sum_{\sigma \in \Sigma} \Pr_{(x,y) \sim M}[f_\sigma(x) \neq f_\sigma(y)] = \frac{2\Pr_{(x,y) \sim M}[f(x) \neq f(y)]}{1 - \lambda},
\]
since given that $f(x) \neq f(y)$, we obtain $f_\sigma(x) \neq f_\sigma(y)$ for precisely two choices of $\sigma$, namely $f(x)$ and $f(y)$. As for the left-hand side, it is equal to
\[
2\sum_{\sigma \in \Sigma} \Pr_{x \sim \mu}[f(x) \neq \sigma] \Pr_{y \sim \mu}[f(y) = \sigma] \geq 2\min_{\sigma \in \Sigma} \Pr_{x \sim \mu}[f(x) \neq \sigma] \sum_{\sigma \in \Sigma} \Pr_{y \sim \mu}[f(y) = \sigma] = 2\min_{\sigma \in \Sigma} \Pr_{x \sim \mu}[f(x) \neq \sigma].
\]
Plugging this in the previous inequality, we conclude that
\[
 \min_{\sigma \in \Sigma} \Pr_{x \sim \mu}[f(x) \neq \sigma] \leq \frac{\Pr_{(x,y) \sim M}[f(x) \neq f(y)]}{1 - \lambda} \leq \frac{\epsilon}{1 - \lambda}. \qedhere
\]
\end{proof}

In our application, $M$ will be a product chain where each factor has full support (a condition which can be replaced by ergodicity) and there are finitely many types of factors. In this case we can bound $\lambda(M)$ irrespective of the number of factors.

\begin{lemma}[Agreement lemma for product chains] \label{lem:agreement-product}
Let $Y,\Sigma$ be finite sets, and let $M_1,\dots,M_t$ be a finite number of symmetric bistochastic $Y \times Y$ matrices with strictly positive entries. There exists $C > 0$ such that the following holds.

Given $\psi\colon [n] \to [t]$, define the distribution $M_\psi$ on $Y^n \times Y^n$ by sampling the $i$'th coordinate independently according to $M_{\psi(i)}$. Let $\mu_\psi$ be the corresponding stationary distribution, obtained by sampling the $i$'th coordinate independently according to $\mu(M_{\psi(i)})$.

If a function $f\colon Y^n \to \Sigma$ satisfies
\[
 \Pr_{(x,y) \sim M_\psi}[f(x) \neq f(y)] \leq \epsilon
\]
then there exists $\sigma \in \Sigma$ such that
\[
 \Pr_{x \sim \mu_\psi}[f(x) \neq \sigma] \leq C\epsilon.
\]
\end{lemma}
\begin{proof}
Given \Cref{lem:agreement-lambda}, it suffices to find $\lambda$ such that $\lambda(M_\psi) \leq \lambda$ for all $\psi$; we can then take $C = 1/(1-\lambda)$.

Since all entries of $M_1,\dots,M_t$ are positive, we can find $\lambda < 1$ such that all eigenvalues of $M_1,\dots,M_t$ other than the top ones are bounded by $\lambda$ in absolute value. It immediately follows that $\lambda(M_\psi) \leq \lambda$.
\end{proof}

\subsection{Arbitrary distributions}
\label{sec:linearity-arbitrary}

In this section we prove \Cref{thm:intro-blr} in full generality, by reduction to the special case $\mu = \pi$ (recall that $\pi$ is the uniform distribution over $P_{m,b}$).

For small enough $q > 0$, we can find a distribution $\nu$ such that
\[
 \mu = q \pi + (1-q) \nu.
\]
Indeed, such a distribution exists whenever $q \leq \min_{w \in P_{m,b}} \mu(w)/\pi(w)$. We furthermore choose $q$ so that $1/q$ is an integer (this will slightly simplify some future argument).

We can sample $x \sim \mu^n$ using a three-step process:
\begin{enumerate}
    \item Let $S \sim \mu_q^n$. This means that $S$ is a subset of $[n]$ chosen by including each $i$ with probability $q$ independently.
    \item If $i \notin S$, sample $x_i$ according to $\nu$.
    \item If $i \in S$, sample $x_i$ according to $\pi$.
\end{enumerate}
If we stop after the first two steps, then the remaining step is a sample from $\pi^S$, which we can analyze using the special case of \Cref{thm:intro-blr}.

\begin{lemma}[Reduction to uniform distribution] \label{lem:blr0}
For every $S \subseteq [m]$ and $\alpha \in P_{m,b}^S$ there exist $b_1(S,\alpha),\dots,b_m(S,\alpha) \in \{0,1\}$ and $A(S,\alpha) \subseteq S$ such that for all $j \in [m]$,
\[
 \Ex_{(S,\alpha) \sim (\mu_q^n, \nu^{S^c})} \Pr_{\pi|_j}[f_j|_{S^c \gets \alpha|_j} \neq \chi_{A(S,\alpha),b_j(S,\alpha)}] = O(\epsilon).
\]
\end{lemma}

Here $\chi_{A,b}(x) = b \oplus \bigoplus_{i \in A} x_i$. We also use $\chi_A = \chi_{A,0}$.

\begin{proof}
For every $S$ and every $\alpha \in P_{m,b}^S$, define
\[
 \epsilon(S, \alpha) = \Pr_{\beta \sim \pi^S}[(f_1|_{S^c \gets \alpha|_1}(\beta|_1), \dots, f_m|_{S^c \gets \alpha|_m}(\beta|_m)) \notin P_{m,b}],
\]
observing that
\[
 \Ex_{(S,\alpha) \sim (\mu_q^n, \nu^{S^c})}[\epsilon(S, \alpha)] = \Pr_{\mu^S}[(f_1,\dots,f_m) \notin P_{m,b}] \leq \epsilon.
\]

Apply the special case of the uniform distribution to every $(S,\alpha)$ to obtain $b_1(S,\alpha),\dots,b_m(S,\alpha) \in \{0,1\}$ and $A(S,\alpha) \subseteq S$ such that for all $j \in [m]$,
\[
 \Pr_{\pi|_j}[f_j|_{S^c \gets \alpha|_j} \neq \chi_{A(S,\alpha),b_j(S,\alpha)}] = O(\epsilon(S,\alpha)).
\]
The result immediately follows.
\end{proof}

The rest of the proof comprises the following steps:
\begin{enumerate}
    \item Using the agreement theorem of~\cite{DFH2025} we show that for each $\alpha \in P_{m,b}^n$ there exists a consensus set $A(\alpha)$ such that typically $A(S,\alpha|_{S^c}) = A(\alpha) \cap S$.
    \item Using \Cref{lem:agreement-product} we show that there exists a consensus set $A$ such that typically $A(\alpha) = A$.
    \item Using \Cref{lem:agreement-product} we show that the functions $f_j \oplus \chi_{A,b_j}$ are close to constant, completing the proof.
\end{enumerate}

\subsubsection{Step 1}

The first step constitutes the following lemma.

\begin{lemma}[Agreement over $S$] \label{lem:blr1}
For every $\alpha \in P_{m,b}^n$ there exists $A(\alpha) \subseteq [n]$ such that
\[
 \Ex_{\alpha \sim \nu^n} \Pr_{S \sim \mu_q}[A(S,\alpha|_{S^c}) \neq A(\alpha)\cap S] = O(\epsilon).
\]
\end{lemma}

The proof will require an agreement theorem essentially proved in~\cite{DFH2025}. The theorem concerns the distribution $\mu_{q,r}^n$, where $0 < q < r < 1$. This is the distribution on triples $(S_1,S_2,T)$ defined as follows:
\begin{itemize}
    \item Sample $T \sim \mu_r^n$.
    \item Sample $S_1 \supseteq T$ so that $S_1 \sim \mu_q^n$.
    \item Sample $S_2 \supseteq T$ so that $S_2 \sim \mu_q^n$.
\end{itemize}
The sampling in the second and third steps can be performed as follows. If $i \in T$, then we always put $i \in S_1$, and otherwise, we put it in $S_1$ with probability $\frac{q-r}{1-r}$.

\begin{theorem}[Agreement theorem] \label{thm:DFH-agreement}
Let $0 < r < q < 1$. Suppose that for every $S \subseteq [n]$ we have a function $\phi_S\colon S \to \Sigma$, where $\Sigma$ is some finite alphabet. If
\[
 \Pr_{(S_1,S_2,T) \sim \mu_{q,r}^n}[\phi_{S_1}|_T \neq \phi_{S_2}|_T] \leq \epsilon
\]
then there exists $\psi\colon [n] \to \Sigma$ such that
\[
 \Pr_{S \sim \mu_q^n}[\phi_S \neq \psi|_S] = O(\epsilon).
\]
\end{theorem}
\begin{proof}
The slice version of this result is \cite[Theorem 3.1]{DFH2025}. This version can be proved using the ``going to infinity'' argument which is used to prove \cite[Theorem 5.4]{DFH2025}.
\end{proof}

We can now prove the lemma.

\begin{proof}[Proof of \Cref{lem:blr1}]
We would like to get into the setting of \Cref{thm:DFH-agreement}, for $q := q$, an appropriate $r$, and $\phi_S := A(S,\alpha|_S)$, for various values of $\alpha$.

For small enough $c > 0$ we can find a distribution $\lambda$ such that
\[
 \pi = c \nu + (1-c) \lambda.
\]
We take $r = (1-c)q$. We will show that
\begin{equation} \label{eq:blr1}
 \Ex_{\alpha \sim \nu^n} \Pr_{(S_1,S_2,T) \sim \mu_{q,r}^n}[A(S_1,\alpha|_{S_1^c}) \cap T \neq A(S_2,\alpha|_{S_2^c}) \cap T] = O(\epsilon).
\end{equation}
Applying \Cref{thm:DFH-agreement} for each $\alpha$ separately then immediately implies the lemma.

In order to show that \Cref{eq:blr1} holds, we use the fact that $f_1|_{S_1 \gets \alpha|_{S_1}}$ is close to $\chi_{A(S_1,\alpha|_{S_1})}$ (up to negation), and similarly $f_1|_{S_2 \gets \alpha|_{S_2}}$ is close to $\chi_{A(S_2,\alpha|_{S_2})}$ (up to negation). Fixing the elements outside of $T$, we obtain that the same function is close to both $\chi_{A(S_1,\alpha|_{S_1}) \cap T}$ and $\chi_{A(S_2,\alpha|_{S_2}) \cap T}$ (up to negation), which can only happen if $A(S_1,\alpha|_{S_1}) \cap T = A(S_2,\alpha|_{S_2}) \cap T$ (since different characters are far from each other).

The first step in this plan is to massage the conclusion of \Cref{lem:blr0} using the equation for $\pi$:
\[
 \Ex_{\substack{(S, \alpha') \sim (\mu_q^n, \nu^{S^c})\\ (T, \alpha'') \sim (\mu_{1-c}^S, \nu^{S \setminus T})}} \min_{B \in \{0,1\}} \Pr_{\lambda|_1^T}[f_1|_{S^c \gets \alpha'|_1, S \setminus T \gets \alpha''|_1} \neq \chi_{A(S,\alpha') \cap T, B}] = O(\epsilon).
\]
(We could determine $B$ explicitly, but this is not necessary.)

An equivalent way to sample $(S,\alpha'),(T,\alpha'')$ is to first sample $S,T$, then sample $\alpha \sim \nu^n$, and take $\alpha' = \alpha|_{S^c}$ and $\alpha'' = \alpha|_{S \setminus T}$. In particular, for $t \in \{1,2\}$,
\[
 \Ex_{\alpha \sim \nu^n} \Ex_{(S_1,S_2,T) \sim \mu_{q,(1-c)q}^n} \min_{B_t \in \{0,1\}} \Pr_{\lambda|_1^T} [f_1|_{T^c \gets \alpha|_{T^c,1}} \neq \chi_{A(S_t,\alpha|_{S_t^c}) \cap T, B_t}] = O(\epsilon).
\]
Combining this for $t = 1$ and $t = 2$ gives
\[
 \Ex_{\alpha \sim \nu^n} \Ex_{(S_1,S_2,T) \sim \mu_{q,(1-c)q}^n} \min_{B_1,B_2 \in \{0,1\}} \Pr_{\lambda|_1^T} [\chi_{A(S_1,\alpha|_{S_1^c}) \cap T, B_1} \neq \chi_{A(S_2,\alpha|_{S_2^c}) \cap T, B_2}] = O(\epsilon).
\]

At this point we use the fact that different characters are far from each other: if $A_1 \neq A_2$ then for all $B_1,B_2$,
\[
 \Pr_{\lambda|_1^T}[\chi_{A_1,B_1} \neq \chi_{A_2,B_2}] \geq \min(\lambda|_1(0), \lambda|_1(1)).
\]
Since $\lambda$ is fixed, this immediately implies \Cref{eq:blr1}, completing the proof.
\end{proof}

\subsubsection{Step 2}

The second step constitutes the following lemma.

\begin{lemma}[Agreement over $\alpha$] \label{lem:blr2}
There exists $A \subseteq [n]$ such that
\[
 \Pr_{\alpha \sim \nu^n}[A(\alpha) \neq A] = O(\epsilon).
\]
\end{lemma}
\begin{proof}
The proof relies on the fact that while $A(\alpha)$ is a function of all of $\alpha$, $A(S,\alpha|_{S^c})$ only depends on part of $\alpha$.
In particular, suppose that we sample two copies of $\nu^n$ by sampling $\alpha \sim \nu^n$ and then obtaining $\alpha'$ by only resampling the coordinates in $S$. Denote this distribution by $\nu(S)$. Since $\alpha|_{S^c} = \alpha'|_{S^c}$, \Cref{lem:blr1} shows that
\[
 \Pr_{(S,\alpha,\alpha') \sim (\mu_q^n, \nu(S))}[A(\alpha) \cap S \neq A(\alpha') \cap S] = O(\epsilon).
\]

In order to proceed, we introduce a conceit:
\[
 \Pr_{\substack{(S,\alpha,\alpha') \sim (\mu_q^n, \nu(S)) \\ T \sim \mu_{1/2}(S)}}[A(\alpha) \cap T \neq A(\alpha') \cap T] \leq
\Pr_{(S,\alpha,\alpha') \sim (\mu_q^n, \nu(S))}[A(\alpha) \cap S \neq A(\alpha') \cap S] = O(\epsilon).
\]
If we marginalize over $S$, the distribution of $(\alpha,\alpha')$, which we denote $\nu'(T)$, becomes the following. First, sample $\alpha \sim \nu^n$. Then, sample $S \supseteq T$ by including any $i \notin T$ with probability $\frac{q/2}{1-q/2}$. Finally, resample all coordinates in $S$. We deduce
\[
 \Pr_{(T,\alpha,\alpha') \sim (\mu_{q/2}^n, \nu'(T))}[A(\alpha) \cap T \neq A(\alpha') \cap T] = O(\epsilon).
\]

In contrast to the distribution $\nu(S)$, which only mixes the coordinates in $S$, the distribution $\nu'(T)$ mixes all coordinates. Moreover, there are only two types of coordinates: those in $T$ (which always get resampled) and the rest (which get resampled with probability $\frac{q/2}{1-q/2}$). This allows us to apply \Cref{lem:agreement-product}, obtaining $A'(T) \subseteq T$ satisfying
\[
 \Ex_{T \sim \mu_{q/2}^n} \Pr_{\alpha \sim \nu^n}[A(\alpha) \cap T \neq A'(T)] = O(\epsilon).
\]

Now it's time for the final trick. Recall that $1/q$ is an integer. Choose $c \colon [n] \to [2/q]$ uniformly at random, and for $t \in [2/q]$, let $T_t(c) = c^{-1}(t)$. Since $T_t \sim \mu_{q/2}^n$ for each $t$, we have
\[
 \Ex_c \Pr_{\alpha \sim \nu^n}[A(\alpha) \cap T_t(c) \neq A'(T_t(c)) \text{ for some }t] = O(\epsilon).
\]

We can find $c$ such that the probability above is $O(\epsilon)$. Define $A$ via $A \cap T_t(c) = A'(T_t(c))$; this makes sense since $T_1(c),\dots,T_{2/q}(c)$ partition $[n]$. The lemma immediately follows.
\end{proof}

\subsubsection{Step 3}

In the final step, we complete the proof of \Cref{thm:intro-blr}.

\begin{proof}[Proof of \Cref{thm:intro-blr}]
The main idea is to ``factor out'' the set $A$ found in \Cref{lem:blr2}. Accordingly, we define functions $h_1,\dots,h_m\colon \{0,1\}^n \to \{0,1\}$ by
\[
 h_j = f_j \oplus \chi_A.
\]
We will complete the proof by showing that $h_j$ is $O(\epsilon)$-close to a constant.

As our starting point, we combine \Cref{lem:blr0,lem:blr1,lem:blr2} to obtain
\[
 \Ex_{(S,\alpha) \sim (\mu_q^n, \nu^{S^c})} \Pr_{\pi|_j^S}[f_j|_{S^c \gets \alpha|_j} \neq \chi_{A \cap S,b_j(S,\alpha)}] = O(\epsilon).
\]
Plugging in $h_j$, we obtain
\[
 \Ex_{(S,\alpha) \sim (\mu_q^n, \nu^{S^c})} \Pr_{\pi|_j^S}[h_j|_{S^c \gets \alpha|_j} \neq \chi_{A \cap S^c,b_j(S,\alpha)}(\alpha|_{S^c,j})] = O(\epsilon).
\]
The expression $\chi_{A \cap S^c,b_j(S,\alpha)}(\alpha|_{S^c,j})$ doesn't depend on the coordinates in $S$, and so
\[
 \Ex_{(S,\alpha) \sim (\mu_q^n, \nu^{S^c})} \Pr_{\beta,\beta' \sim \pi|_j^S}[h_j|_{S^c \gets \alpha|_j}(\beta) \neq h_j|_{S^c \gets \alpha|_j}(\beta')] = O(\epsilon).
\]

Let $\gamma \in P_{m,b}^n$ be the vector defined by $\gamma|_{S^c} = \alpha$ and $\gamma|_S = \beta$, and let $\gamma' \in P_{m,b}^n$ be defined similarly with $\gamma'|_S = \beta'$. The inputs to $h_j$ in the formula above are thus $\gamma|_j$ and $\gamma'|_j$.

We can sample $\gamma,\gamma'$ as follows. First, sample $\gamma \in \mu^n$. Next, sample $S$ given $\gamma$; the (non-zero) probability that $i \in S$ depends only on $\gamma_i$, and can be calculated using Bayes' law. Finally, $\gamma'$ is obtained by resampling the coordinates in $S$ according to $\pi$.
This means that the distribution of $(\gamma,\gamma')$ corresponds to some product chain in the sense of \Cref{lem:agreement-product}, with a single type of factor. The same holds for $(\gamma|_j,\gamma'|_j)$, and so \Cref{lem:agreement-product} implies that there exists $b_j \in \{0,1\}$ such that
\[
 \Pr_{\mu^n}[h_j \neq b_j] = O(\epsilon).
\]

Finally, we take $g_j = \chi_{A,b_j}$. Observe that $\Pr_{\mu|_j^n}[g_j \neq f_j] = O(\epsilon)$, and so
\[
 \Pr_{\mu^n}[(g_1,\dots,g_m) \notin P_{m,b}] = O(\epsilon).
\]
Whether $(g_1,\dots,g_m) \in P_{m,b}$ or not depends only on $b_1,\dots,b_m$ (the dependence on the input cancels out). We can assume without loss of generality that $\epsilon$ is small enough to make the right-hand side smaller than~$1$. This implies that $(g_1,\dots,g_m)$ is a generalized polymorphism of $P_{m,b}$, completing the proof.
\end{proof}

\section{Intersecting families} \label{sec:intersecting}

In this section we provide an alternative proof \Cref{thm:intro-friedgut-regev}, originally due to Friedgut and Regev~\cite{FR18} (who improved earlier results of Dinur, Friedgut, and Regev~\cite{DFR08,DR09}).

\thmintrofriedgutregev*

Instead of working directly with functions on $\binom{[n]}{pn}$, Friedgut and Regev first reduce \Cref{thm:intro-friedgut-regev} to the following statement.

\begin{theorem}[Fractional Friedgut--Regev] \label{thm:friedgut-regev-fractional}
Fix $p \in (0,1/2)$. Let $P_{\NAND} = \{ (0,0),(1,0),(0,1) \}$, and let $\mu$ be the following distribution: $\mu(0,0) = 1-2p$, $\mu(1,0) = \mu(0,1) = p$.
For every $\epsilon > 0$ there exist $\delta,M > 0$ such that the following holds.

Suppose that $f_1,f_2\colon \{0,1\}^n \to [0,1]$ satisfy
\[
 \Ex_{(x^{(1)},x^{(2)}) \sim \mu^n}[f_1(x^{(1)}) f_2(x^{(2)})] \leq \delta.
\]
Then there exist functions $g_1,g_2\colon \{0,1\}^n \to \{0,1\}$, depending on at most $M$ coordinates, such that $(g_1,g_2)$ is a generalized polymorphism of $P_{\NAND}$, and moreover for $j \in \{1,2\}$,
\[
 \Ex_{x^{(j)} \sim \mu|_j^n}[(1-g_j(x^{(j)})) f_j(x^{(j)})] \leq \epsilon.
\]

Furthermore, if $f_1 = f_2$ then $g_1 = g_2$.
\end{theorem}

Notice that in this statement, the functions $f_1,f_2$ are not necessarily Boolean, but they take values in the interval $[0,1]$.

Let us briefly explain how \Cref{thm:intro-friedgut-regev} follows from \Cref{thm:friedgut-regev-fractional}.
Given $\mathcal{F}$, which we identify with the corresponding Boolean function, we define
\[
 f_1(x) = f_2(x) =
 \begin{cases}
     \Ex_{\substack{x' \leq x \\ |x'| = k}}[\mathcal{F}(x')] & \text{if } |x| \geq k, \\
     0 & \text{otherwise}.
 \end{cases}
\]
\cite[Lemma 7.3]{FR18} shows that the assumption of \Cref{thm:intro-friedgut-regev} implies that of \Cref{thm:friedgut-regev-fractional}. We apply \Cref{thm:friedgut-regev-fractional}, and convert $g_1 = g_2$ to a family $\mathcal{G}$ in the natural way. \cite[Lemma 7.4]{FR18} shows that the conclusion of \Cref{thm:friedgut-regev-fractional} implies the conclusion of \Cref{thm:intro-friedgut-regev}.

The proof of \Cref{thm:friedgut-regev-fractional} is very similar to the proof of \Cref{thm:intro-monotone}. First, we need versions of \Cref{thm:it-aint-over} (It Ain't Over Till It's Over) and \Cref{thm:jones-was-tree} (Jones' regularity lemma) for functions taking values in $[0,1]$.

The original proof of \Cref{thm:it-aint-over}~\cite{MOO10} was in fact formulated for functions taking values in $[0,1]$. (Note that our definition of regularity makes sense for arbitrary real-valued functions.)

\begin{theorem}[Fractional It Ain't Over Till It's Over] \label{thm:it-aint-over-fractional}
\Cref{thm:it-aint-over} holds for functions $f\colon \{0,1\}^n \to [0,1]$ (with the same parameters).
\end{theorem}

A simple modification of the proof of Jones' regularity lemma, which we present in \Cref{sec:regularity}, extends it to the same setting.

\begin{theorem}[Fractional Jones' regularity lemma] \label{thm:jones-was-tree-fractional}
\Cref{thm:jones-was-tree} holds for functions $f\colon \{0,1\}^n \to [0,1]$ (with the same parameters).
\end{theorem}

Upon inspection, the proof of \Cref{lem:monotone-counting-NAND} (Counting lemma for NAND) translates to the following statement.

\begin{lemma}[Fractional counting lemma for NAND] \label{lem:monotone-counting-NAND-fractional}
Fix $p \in (0, 1/2)$, and let $P_{\NAND},\mu$ be as in \Cref{thm:friedgut-regev-fractional}.
For every $\epsilon > 0$ there exists a constant $C = C(P,\mu)$ such that the following holds for $d = \Theta(\log(1/\epsilon))$, $\tau = \Theta(\epsilon^C)$, and $\gamma = \Theta(\epsilon^{2C})$.

Let $\phi_1,\dots,\phi_m\colon \{0,1\}^n \to [0,1]$ be functions such that $\phi_j$ is $(d,\tau)$-regular with respect to $\mu|_j$ and $\Ex_{\mu|_j}[\phi_j] \geq \epsilon$. Then
\[
 \Ex_{(y^{(1)},\dots,y^{(2)}) \sim \mu^n}[\phi_1(y^{(1)}) \phi_2(y^{(2)})] \geq \gamma.
\]
\end{lemma}
\begin{proof}[Proof sketch]
In the proof of \Cref{lem:monotone-counting-NAND} we define a random restriction $\rho \sim \nu^n$, where $\nu$ is supported on inputs in $\{0,1,*\}^m$ with at most one free coordinate. We show that for an appropriate choice of $d = \Theta(\log(1/\epsilon))$ and $\tau,\delta_1,\delta_2 = \Theta(1/\epsilon^C)$,
\[
 \Pr_{\rho}[\Ex_{\mu|_j}[\phi_j|_{\rho|_j}] \geq \delta_j] \geq 1 - \epsilon \text{ for } j \in \{1,2\}.
\]
If we sample $(y^{(1)},y^{(2)})|\rho$ then $\phi_1(y^{(1)}),\phi_2(y^{(2)})$ are independent since they depend on disjoint sets of coordinates. Therefore
\[
 \Ex_{(y^{(1)},\dots,y^{(2)}) \sim \mu^n}[\phi_1(y^{(1)}) \phi_2(y^{(2)})] \geq (1-\epsilon)\delta_1\delta_2. \qedhere
\]
\end{proof}

We can now prove \Cref{thm:friedgut-regev-fractional}, closely following the proof of \Cref{thm:intro-monotone}.

\begin{proof}[Proof of \Cref{thm:friedgut-regev-fractional}]
Apply \Cref{lem:monotone-counting-NAND-fractional} with $\epsilon := \epsilon/2$ to obtain $d,\tau,\gamma$.

We apply \Cref{thm:jones-was-tree-fractional} with $\epsilon := \epsilon/2$ and the parameters $d,\tau$ to $f_1,f_2$ to obtain a set $J$ of size $M$. 
We define the functions $g_1,g_2\colon \{0,1\}^n \to \{0,1\}$ as follows. For each $x \in \{0,1\}^J$,
 \[
  g_j|_{J \gets x} = \begin{cases}
      0 & \text{if $f_j|_{J \gets x}$ is not $(d,\tau)$-regular or $\Ex_{\mu|_j}[f_j|_{J \gets x}] \leq \epsilon/2$}, \\
      1 & \text{otherwise}.
  \end{cases}
 \]

If we sample $x$ according to $\mu|_j$ then according to Jones' regularity lemma, $f_j|_{J \gets x}$ is $(d,\tau)$-regular with probability at least $1-\epsilon/2$. This shows that $\Ex_{\mu|_j}[(1-g_j) f_j] \leq \epsilon/2 + \epsilon/2 = \epsilon$.

It remains to show that $(g_1,g_2)$ is a generalized polymorphism of $P_\NAND$ for small enough $\delta$. If this is not the case, then there exists an assignment $(x^{(1)},x^{(2)}) \in P^J$ such that $(g_1|_{J \gets x^{(1)}},g_2|_{J \gets x^{(2)}})$ is not a generalized polymorphism of $P_\NAND$, which implies that $f_1|_{J \gets x^{(1)}},f_2|_{J \gets x^{(2)}}$ are both regular and have expectation at least $\epsilon/2$. \Cref{lem:monotone-counting-NAND-fractional} thus implies that $\Ex_\mu[f_1|_{J \gets x^{(1)}} f_2|_{J \gets x^{(2)}}] \geq \gamma$, which is ruled out by defining $\delta = \min(\mu)^M\gamma/2$.
\end{proof}

The proofs of \Cref{thm:friedgut-regev-fractional,thm:intro-friedgut-regev} extend to other predicates. We leave this to future work.

\section{General alphabets} \label{sec:general-alphabets}

In this section we prove \Cref{thm:intro-alphabet}, an analog of \Cref{thm:intro-main} for some predicates over larger alphabets.

\thmintroalphabet*

The proof follows the general outline of the proof of \Cref{thm:intro-main}, but is much simpler. First, the counting lemma is direct rather than inductive. Second, there is no need to accommodate affine relations.

Before starting the proof proper, we need to generalize the notion of regularity to the setting of functions $f\colon \Sigma^n \to \Sigma$. The original proof of It Ain't Over Till It's Over~\cite{MOO10} actually works for arbitrary alphabets. Given a function $f\colon \Sigma^n \to \{0,1\}$ and a distribution $\mu$ over $\Sigma$ with full support, the Efron--Stein decomposition is the unique decomposition
\[
 f = \sum_S f_S
\]
in which $f_S$ depends only on the coordinates in $S$, the functions $f_S$ are orthogonal, and $f_S$ has zero expectation if we fix the values of coordinates in some set $T \not\supseteq S$. The definition of low-degree influences readily extends:
\[
 \Inf_i[f^{\leq d}] = \sum_{\substack{|S| \leq d \\ i \in S}} \|f_S\|^2,
\]
where the norm is computed according to $\mu$.

This prompts the following definition of regularity.

\begin{definition}[Regularity for arbitrary alphabets] \label{def:alphabet-regularity}
Let $\Sigma$ be a finite alphabet of size at least~$2$, let $\mu$ be a distribution on $\Sigma$ with full support, and let $d \in \mathbb{N}$ and $\tau > 0$.

A function $f\colon \Sigma^n \to \Sigma$ is $(d,\tau)$-regular with respect to $\mu$ if $\Inf_i[(f^{=\sigma})^{\leq d}]$ for all $i \in [n]$ and $\sigma = \Sigma$, where
\[
 f^{=\sigma}(x) = \begin{cases}
     1 & \text{if } f(x) = \sigma, \\
     0 & \text{otherwise}.
 \end{cases}
\]
\end{definition}

With this definition, we can extend both It Ain't Over Till It's Over and Jones' regularity lemma.

\begin{theorem}[It Ain't Over Till It's Over for arbitrary alphabets] \label{thm:it-aint-over-alphabet}
For every alphabet $\Sigma$, full support distribution $\mu$, $q \in (0,1)$ and $\epsilon > 0$ the following holds for some constant $C$ and $d = \Theta(\log(1/\epsilon))$, $\tau = \Theta(\epsilon^C)$, and $\delta = \Theta(\epsilon^C)$.

Let $\rho$ be a random restriction obtained by sampling each coordinate $i \in [n]$ independently according to the following law: with probability $1-q$, draw a sample from $\mu$, and otherwise, draw $*$.

If $f\colon \Sigma^n \to \Sigma$ is $(d,\tau)$-regular with respect to $\mu$ then for every $\sigma \in \Sigma$ such that $\Pr_{\mu}[f = \sigma] \geq \epsilon$ we have
\[
 \Pr_\rho[\Pr[f|_\rho = \sigma] \geq \delta] \geq 1-\epsilon.
\]
\end{theorem}
\begin{proof}
Follows immediately from~\cite{MOO10} by considering the functions $f^{=\sigma}$.
\end{proof}

Jones' regularity lemma can also be extended, as we show in \Cref{sec:regularity}.

\begin{theorem}[Jones' regularity lemma for arbitrary alphabets] \label{thm:jones-alphabet}
For every alphabet $\Sigma$, $m \in \mathbb{N}$, full support distributions $\mu_1,\dots,\mu_m$, and every $\epsilon, \tau > 0$, $d \in \mathbb{N}$ the following holds for some function $M \in \mathbb{N}$.

For all functions $f_1,\dots,f_m\colon \Sigma^n \to \Sigma$ there exists a set $J$ of size at most $M$ such that for all $j$,
\[
 \Pr_{x \sim \mu_j}[f|_{J \gets x} \text{ is $(d,\tau)$-regular with respect to $\mu_j$}] \geq 1 - \epsilon.
\]
\end{theorem}

While we can prove a counting lemma in the style of \Cref{lem:monotone-counting} or \Cref{lem:counting}, the counting lemma that is useful here is similar to the one implicitly used to prove \Cref{lem:rounding-poly}. We prove this lemma in \Cref{sec:alphabet-counting}, where we also prove an analog of \Cref{lem:rounding-soundness}. Combining these with Jones' regularity lemma, we complete the proof of \Cref{thm:intro-alphabet} in \Cref{sec:alphabet-main}.

For the rest of this section, we fix $P$ and $\mu$.

\subsection{Counting lemma} \label{sec:alphabet-counting}

The counting lemma that we prove uses a restriction similar to the one appearing in \Cref{sec:main-rounding}. In that section, we had to distinguish between flexible coordinates and inflexible coordinates. In our case, all coordinates are flexible by assumption.

\begin{definition}[Restriction] \label{def:alphabet-restriction}
For every $j \in [m]$, let $w_{(j,*)}$ be a partial input, missing only the $j$'th coordinate, such that all of its completions $w_{(j,\sigma)}$ belong to $P$.

The distribution $\nu$ is a distribution on $Q := P \cup \{ w_{(j,*)} : j \in [m] \}$ defined as follows, for a small enough $q > 0$:
\begin{itemize}
    \item For $w \in P$, sample $w$ with probability
    \[
     \mu(w) - q \sum_{(j,\sigma)\colon w = w_{(j,\sigma)}} \mu|_j(\sigma).
    \]
    \item For every $j \in [m]$, sample $w_{(j,*)}$ with probability $q$.
\end{itemize}
Concretely, it suffices to take $q = \min(\mu)/m$, where $\min(\mu) = \min_{w \in P} \mu(w)$.

Given $\rho \in Q$, the distribution $\mu|\rho$ is obtained by sampling the missing coordinate in $w_{(j,*)}$ using $\mu|_j$.

By construction, if $\rho \sim \nu$ and $x \sim \mu|\rho$ then $x \sim \mu$. Moreover, the marginal distribution of $\rho_j$ given $\rho_j \neq *$ is $\mu|_j$.
\end{definition}

\begin{lemma}[Counting lemma] \label{lem:alphabet-counting}
Let $\phi_1,\dots,\phi_m\colon \Sigma^n \to \Sigma$, and let $\rho \in Q^n$. Let $w \in \Sigma^m$ be such that for each $j \in [m]$,
\[
 \Pr_{\mu|_j}[\phi_j|_{\rho|_j} = w_j] \geq \epsilon.
\]
Then
\[
 \Pr_{\mu|\rho}[(\phi_1,\dots,\phi_m) = w] \geq \epsilon^m.
\]
\end{lemma}
\begin{proof}
Since $\rho$ is supported on inputs with at most one free coordinate, the events $\phi_j|_{\rho|_j} = w_j$ are independent. The lemma immediately follows.
\end{proof}

We will eventually choose $\rho$ and $\delta$ (a bound on $\Pr_\mu[(f_1,\dots,f_m) \notin P]$) so that the counting lemma implies that any $w$ satisfying the conditions in the lemma belongs to $P$. This suggests the following rounding procedure.

\begin{definition}[Rounding] \label{def:alphabet-rounding}
Let $\phi_1,\dots,\phi_m \in \Sigma^n \to \Sigma$ and let $\rho \in Q^n$.

For a parameter $\epsilon \in (0,1)$, we define $\round_j^{\rho,\epsilon}(\phi_j)\colon \Sigma^n \to \Sigma$ as follows. First, let
\[
 \Sigma_{\geq \epsilon} := \{ \sigma \in \Sigma : \Pr_{\mu|_j}[\phi_j|_{\rho_j} = \sigma] \geq \epsilon \},
\]
and choose $\sigma_0 \in \Sigma_{\geq \epsilon}$ arbitrarily (say the symbol with the largest probability). Then
\[
 \round_j^{\rho,\epsilon}(\phi_j)(x) =
 \begin{cases}
     \phi_j(x) & \text{if } \phi_j(x) \in \Sigma_{\geq \epsilon}, \\
     \sigma_0 & \text{otherwise}.
 \end{cases}
\]
\end{definition}

The following lemma bounds the error in rounding, and roughly corresponds to \Cref{lem:rounding-soundness}.

\begin{lemma}[Rounding lemma] \label{lem:alphabet-rounding}
For every $\epsilon > 0$ there exist $d \in \mathbb{N}$ and $\tau,\eta > 0$ such that the following holds.

Let $\phi_1,\dots,\phi_m\colon \Sigma^n \to \Sigma$. If $\phi_j$ is $(d,\tau)$-regular with respect to $\mu|_j$ for all $j$ then for all $j$,
\[
 \Ex_{\rho \sim \nu^n}\bigl[\Pr_{\mu|_j}[\round_j^{\rho,\eta}(\phi_j) \neq \phi_j]\bigr] \leq \epsilon.
\]
\end{lemma}
\begin{proof}
Apply \Cref{thm:it-aint-over-alphabet} (It Ain't Over Till It's Over) with $\Sigma,\mu,q$ and $\epsilon := \epsilon/|\Sigma|$ to obtain $d,\tau,\delta$. We take $\eta := \delta$.

Recall the definition of $\Sigma_{\geq \eta}$ in \Cref{def:alphabet-rounding}. By definition,
\[
 \Pr_{\mu|_j}[\round_j^{\rho,\eta} \neq \phi_j] =
 \sum_{\sigma \notin \Sigma_{\geq \eta}} \Pr_{\mu|_j}[\phi_j = \sigma].
\]
Therefore
\[
 \Ex_{\rho \sim \nu^n}\bigl[\Pr_{\mu|_j}[\round_j^{\rho,\eta} \neq \phi_j]\bigr] =
 \sum_{\sigma \in \Sigma} \Pr[\sigma \notin \Sigma_{\geq \eta}] \Pr_{\mu|_j}[\phi_j = \sigma].
\]

For each $\sigma \in \Sigma$, we consider two cases. If $\Pr_{\mu|_j}[\phi_j = \sigma] < \epsilon/|\Sigma|$ then the summand is clearly at most $\epsilon/|\Sigma|$. Otherwise, $\Pr[\sigma \notin \Sigma_{\geq \eta}] \leq \epsilon/|\Sigma|$ by \Cref{thm:it-aint-over-alphabet}. In both cases, the summand is at most $\epsilon/|\Sigma|$. Summing over all $\sigma \in \Sigma$ yields the result.
\end{proof}

\subsection{Main result} \label{sec:alphabet-main}

We are now ready to prove \Cref{thm:intro-alphabet}.

\begin{proof}[Proof of \Cref{thm:intro-alphabet}]
We prove the theorem with an error probability of $O(\epsilon)$ rather than $\epsilon$.

Apply \Cref{lem:alphabet-rounding} with $\epsilon$ to obtain $d,\tau,\eta$. Apply \Cref{thm:jones-alphabet} with alphabet $\Sigma$, distributions $\mu|_1,\dots,\mu|_m$, and parameters $\epsilon,d,\tau$ to obtain $M$ and a set $J$ of size at most $M$ such that for all $j$,
\[
 \Pr_{x \sim \mu_j^J}[f|_{J \gets x} \text{ is $(d,\tau)$-regular with respect to $\mu_j$}] \geq 1-\epsilon.
\]

As in the proof of \Cref{thm:intro-main}, we say that a subfunction $f_j|_{J \gets x}$ is \emph{good} if there exist $(x^{(1)},\dots,x^{(m)}) \in P^J$ such that $x^{(j)} = x$ and $f_k|_{J \gets x^{(k)}}$ is $(d,\tau)$-regular for all $k$. A simple argument (written down explicitly as \Cref{lem:good-most}) shows that for all $j$,
\[
 \Pr_{x \sim \mu|_j}[f_j|_{J \gets x} \text{ is good}] \geq 1-m\epsilon.
\]

With hindsight, define
\[
 \delta = \min(\mu)^M \eta^m/3.
\]

We have
\[
 \Ex_{\rho \sim \nu^{J^c}} \Ex_{(x^{(1)},\dots,x^{(m)}) \sim \mu^J} \Pr_{\mu^{J^c}|\rho}[(f_1|_{J \gets x^{(1)}},\dots,f_m|_{J \gets x^{(m)}}) \notin P] = \Pr_{\mu^n}[(f_1,\dots,f_m) \notin P] \leq \delta.
\]
Therefore with probability at least $1/2$ over the choice of $\rho$,
\[
 \Ex_{(x^{(1)},\dots,x^{(m)}) \sim \mu^J} \Pr_{\mu^{J^c}|\rho}[(f_1|_{J \gets x^{(1)}},\dots,f_m|_{J \gets x^{(m)}}) \notin P] \leq 2\delta < \min(\mu)^M \eta^m.
\]
In this case, we say that $\rho$ is \emph{good}.

If $\rho$ is good then for any $(x^{(1)},\dots,x^{(m)}) \in P^J$,
\[
 \Pr_{\mu|\rho}[(f_1|_{J \gets x^{(1)}},\dots,f_m|_{J \gets x^{(m)}}) \notin P] < \eta^M.
\]
Applying \Cref{lem:alphabet-counting}, this shows that $(\round_1^{\rho,\eta}(f_1|_{J \gets x^{(1)}}), \dots, \round_m^{(\rho,\eta}(f_m|_{J \gets x^{(m)}}))$ is a generalized polymorphism of $P$.
Accordingly, we define $g_1^\rho,\dots,g_m^\rho$ by
\[
 g_j^\rho|_{J \gets x} = \round_j^{(\rho,\eta)}(f_j|_{J \gets x}).
\]
For every good $\rho$, the functions $(g_1,\dots,g_m)$ are a generalized polymorphism of $P$.

If $f_j|_{J \gets x}$ is good then \Cref{lem:alphabet-rounding} shows that
\[
 \Ex_{\substack{\rho \sim \nu^{J^c} \\ \rho \text{ good}}}\bigl[\Pr_{\mu|_j^{J^c}}[g_j^\rho|_{J \gets x} \neq f_j|_{J \gets x}]\bigr] \leq \epsilon/\Pr[\rho\text{ is good}] \leq 2\epsilon.
\]
Since the probability that $f_j|_{J \gets x}$ is not good is at most $m\epsilon$, it follows that
\[
 \sum_{j=1}^m \Ex_{\substack{\rho \sim \nu^{J^c} \\ \rho \text{ good}}} \bigl[\Pr_{\mu_j^n}[g_j^\rho \neq f_j]\bigr] \leq m(m+2)\epsilon.
\]
In particular, we can find $\rho$ for which this sum is at most $m(m+2)\epsilon$. Taking $g_j = g_j^\rho$ for all $j$ completes the proof.
\end{proof}

\section{Regularity lemmas} \label{sec:regularity}

In this section we prove the various versions of Jones' regularity lemma: \Cref{thm:jones-was-tree,thm:jones-junta,thm:jones-was-tree-fractional,thm:jones-alphabet}.

We derive all of them from the following general version.

\begin{theorem}[Jones' regularity lemma] \label{thm:jones-general-junta}
For every alphabet $\Sigma$, every $m \in \mathbb{N}$, every full support distributions $\mu_1,\dots,\mu_m$, and every $\epsilon, \tau > 0$, $d \in \mathbb{N}$ the following holds for some function $\mathcal{M}\colon \mathbb{N} \to \mathbb{N}$.

For all functions $f_1,\dots,f_m\colon \Sigma^n \to [0,1]$ and set $J_0 \subseteq [n]$ there exists a set $J \supseteq J_0$ of size at most $\mathcal{M}(|J|)$ such that for all $j$,
\[
 \Pr_{x \sim \mu_j}[f_j|_{J \gets x} \text{ is $(d,\tau)$-regular with respect to $\mu_j$}] \geq 1 - \epsilon.
\]
\end{theorem}

\medskip

The proof proceeds via a different notion of regularity, which uses noisy influences.

\begin{definition}[Noisy influences] \label{def:noisy-influences}
For $\rho \in (0,1)$, let $N_\rho$ be the distribution of pairs $x,y \in \Sigma^n$ sampled as follows. We sample $x \sim \mu^n$. We sample $y$ by resampling each coordinate of $x$ with probability $1-\rho$. The noise stability of a function $f\colon \Sigma^n \to \mathbb{R}$ is defined as
\[
 \Stab_\rho[f] = \Ex_{(x,y) \sim N_\rho}[f(x) f(y)] = \sum_S \rho^{|S|} \|f_S\|^2,
\]
where $f_S$ are the components of the Efron--Stein decomposition of $f$, and the norm is computed according to $\mu$.

For a coordinate $i$, let $E_i f$ be obtained by averaging over the coordinate $i$:
\[
 E_i f(x) = \Ex_{a \sim \mu}[f(x|_{i \gets a})] = \sum_{i \notin S} f_S.
\]
The noisy influences of $f$ are
\[
 \Inf_i^\rho[f] = \Stab_\rho[f - E_i f] = \sum_{i \in S} \rho^{|S|} \|f_S\|^2.
\]

We say that $f$ is $(\rho,\tau)$-noisy-regular if $\Inf_i^\rho[f] \leq \tau$ for all $i$.
\end{definition}

We derive \Cref{thm:jones-general-junta} from the following version for noisy influences.

\begin{theorem}[Jones' regularity lemma for noisy influences] \label{thm:jones-general-noisy-junta}
For every alphabet $\Sigma$, every $m \in \mathbb{N}$, every full support distributions $\mu_1,\dots,\mu_m$, and every $\epsilon, \tau > 0$, $d \in \mathbb{N}$ the following holds for some function $\mathcal{M}\colon \mathbb{N} \to \mathbb{N}$.

For all functions $f_1,\dots,f_m\colon \Sigma^n \to [0,1]$ and every $J_0 \subseteq [n]$ there exists a set $J \supseteq J_0$ of size at most $\mathcal{M}(|J_0|)$ such that for all $j$,
\[
 \Pr_{x \sim \mu_j}[f_j|_{J \gets x} \text{ is $(\rho,\tau)$-noisy-regular with respect to $\mu_j$ for all $j$}] \geq 1 - \epsilon.
\]
\end{theorem}

The proof of \Cref{thm:jones-general-noisy-junta} uses a potential argument, and employs decision trees as intermediate representations. For a decision tree $T$, we define
\[
 \Phi_{f,\mu}^\rho(T) = \Ex_{\ell \in T}[\Stab_\rho[f|_\ell]],
\]
where $\ell$ is a random leaf of $T$ sampled using $\mu$. The potential function we use is the sum of these potentials for all $f_j$:
\[
 \Phi(T) = \sum_j \Phi_{f_j,\mu_j}^\rho(T).
\]

All properties of this potential function follow from the following simple calculation, which corresponds to splitting a node in the tree.

\begin{lemma}[Splitting a node] \label{lem:splitting}
For every $i$ we have
\[
 \Ex_{a \sim \mu}[\Stab_\rho[f|_{i \gets a}]] = \Stab_\rho[f] + \frac{1-\rho}{\rho} \Inf_i^\rho[f].
\]
\end{lemma}
\begin{proof}
The definition of noise stability shows that
\[
 \Stab_\rho[f] = \rho \Ex_{a \sim \mu}[\Stab_\rho[f|_{i \gets a}]] +
 (1-\rho) \Ex_{\substack{a,b \sim \mu \\ (x',y') \sim N_\rho}}[f|_{i \gets a}(x') f|_{i \gets b}(y')].
\]

Similarly, the noisy influence is
\[
 \Inf_i^\rho[f] = \Ex_{\substack{(x,y) \sim N_\rho \\ a,b \sim \mu}}[(f(x) - f(x|_{i \gets a}))(f(y) - f(y|_{i \gets b})] = 
 \Stab_\rho[f] - \Ex_{\substack{a,b \sim \mu \\ (x',y') \sim N_\rho}}[f|_{i \gets a}(x') f|_{i \gets b}(y')],
\]
where the second term is the result of summing three identical terms with different signs. Substituting the earlier formula gives
\[
 \Inf_i^\rho[f] = \rho \Ex_{a \sim \mu}[\Stab_\rho[f|_{i \gets a}]] -
 \rho \Ex_{\substack{a,b \sim \mu \\ (x',y') \sim N_\rho}}[f|_{i \gets a}(x') f|_{i \gets b}(y')].
\]
The lemma immediately follows.
\end{proof}

We can now prove \Cref{thm:jones-general-noisy-junta}.

\begin{proof}
For a set $J$, let $C(J)$ be the decision tree querying all coordinates in $J$.

We construct a sequence of sets $J_t$ as follows. The start point is $J_0$. If $J_t$ doesn't satisfy the conclusion of the lemma for $f_j$, by \Cref{lem:splitting} we can add one more level to $C(J_t)$ to obtain a tree $T'$ such that
\[
 \Phi(T') \geq \Phi(C(J_t)) + \frac{1-\rho}{\rho}\epsilon\tau.
\]
Let $J_{t+1}$ consist of all variables appearing in $T'$. We can obtain a tree with the same set of leaves as $C(J_{t+1})$ by splitting nodes of $T'$, and so \Cref{lem:splitting} implies that
\[
 \Phi(C(J_{t+1})) \geq \Phi(T') \geq \Phi(C(J_t)) + \frac{1-\rho}{\rho}\epsilon\tau.
\]

Since $\Stab_\rho[f] \leq \|f\|^2$ for every function $f$, it is easy to check that $0 \leq \Phi(T) \leq m$ for all $T$. In particular, the process above must terminate after at most $m/\epsilon\tau \cdot \rho/(1-\rho)$ steps.
Since $|J_{t+1}| \leq |J_t| + |\Sigma|^{|J_t|}$, the final size of the resulting set $J$ can be bounded independently of $n$.
\end{proof}

We deduce \Cref{thm:jones-general-junta} by relating low-degree influences and noisy influences.

\begin{proof}[Proof of \Cref{thm:jones-general-junta}]
If a function $g\colon \Sigma^n \to [0,1]$ is $(\rho,\tau)$-noisy-regular then for every $\rho \in (0,1)$ and every $i \in [n]$ we have
\[
 \Inf_i[g^{\leq d}] = \sum_{\substack{|S| \leq d \\ i \in S}} \|f_S\|^2 \leq \rho^{-d} \sum_{i \in S} \rho^{|S|} \|f_S\|^2 = \rho^{-d} \Inf_i^\rho[g].
\]

We choose $\rho = 1 - 1/d$ and apply \Cref{thm:jones-general-noisy-junta} with $\tau := \rho^d \tau$ to complete the proof, observing that $\rho^d = \Theta(1)$.
\end{proof}

The proof makes it clear that the dependence of $\mathcal{M}$ on the parameters is quite bad (of tower type).

\section{Open questions} \label{sec:open-questions}

We would like to highlight the following open questions:
\begin{enumerate}
    \item Extend \Cref{thm:intro-alphabet} to cover all predicates $P \subseteq \Sigma^m$.

    The simplest predicate that our techniques cannot handle is $\{(a,b,c) \in \mathbb{Z}_3 : a + b + c \neq 0\}$.
    
    \item Prove \Cref{thm:intro-main} with the additional guarantee that if $f_i = f_j$ and $\mu|_i = \mu|_j$ then $g_i = g_j$.

    This guarantee follows from \Cref{thm:intro-main} when all generalized polymorphisms $g_1,\dots,g_m$ of $P$ are such that if $\mu|_i = \mu_j$ then either $g_i = g_j$ or $\Pr_{\mu|_i}[g_i \neq g_j] = \Omega(1)$. This is the case for the predicates $P_{m,b}$ considered in \Cref{thm:intro-blr}. We are also able to provide this guarantee in the monotone setting (\Cref{thm:intro-monotone}).

    There are other situations in which we can prove this guarantee as a corollary of \Cref{thm:intro-main}. One example is the predicate $P_\land = \{(a,b,a \land b) : a,b \in \{0,1\}\}$. In this case the generalized polymorphisms $g_1,g_2,g_3$ come in two types: (i) $g_1 = g_2 = g_3 = \bigwedge_{i \in S} x_i$ for some $S$; (ii) $g_1 = g_3 = 0$ or $g_2 = g_3 = 0$. In the first case, we automatically get the guarantee. In the second case, say $g_1 = g_3 = 0$, the only claim which doesn't automatically hold is that if $f_1 = f_2$ and $\mu|_1 = \mu|_2$ then $g_1 = g_2$. In this case $g_2$ is close to $g_1$ (since $f_1 = f_2$ and $g_1,g_2$ are close to $f_1,f_2$), and so we can set $g_2 \equiv 0$ to satisfy the guarantee while maintaining the other properties in the theorem.

    Many predicates satisfy a similar but more restricted type of guarantee: every generalized polymorphism $(g_1,\dots,g_m)$ is such that either (i) all functions are of the form $x_i$ or $1 - x_i$ (for the same $i$), or (ii) some of the functions are constant, and these constants constitute a ``certificate'' for the predicate. This case can be handled just as the case of $P_\land$. A characterizations of this type of behavior can be found in~\cite{Filmus2025}.

    \item Determine the optimal relation between $\epsilon$ and $\delta$ in \Cref{thm:intro-main}. We conjecture that the optimal dependence is polynomial or even linear. This holds in the case of the predicates $P_{m,b}$ of \Cref{thm:intro-blr}, and also for many functional predicates (see~\cite[Appendix D]{CFMMS2022} for an illustrative example).
\end{enumerate}

\addcontentsline{toc}{section}{Bibliography}

\bibliographystyle{alpha}
\bibliography{biblio}

\end{document}